\documentclass[reqno]{amsart}
\numberwithin{equation}{section} 
\usepackage{graphicx}
\usepackage{wrapfig}

\newtheorem{theorem}{\sc Theorem}[section]
\newtheorem{lemma}{\sc Lemma}[section]

\newtheorem{remark}{\sc Remark}

\def\bdy #1{\partial #1}
\def\cls #1{\overline {#1}}

\def\bbR{{\mathbb R}}

\def\div{\operatorname{div}}

\def\curl{\operatorname{curl}}
\def\id{{\text{Id}}}
\def\supp{{\text{supp}}}
\def\bA{{\bar{A}}}
\def\ak{{a_\kappa}}
\def\ba{{\bar{a}_\kappa}}
\def\gk{{g_\kappa}}
\def\bg{{\bar{g}}}
\def\be{{\bar{\eta}}}
\def\bek{{\bar{\eta}_\kappa}}
\def\nk{{n_\kappa}}
\def\bn{\bar{n}_\kappa}

\def\bv{\bar{v}}
\def\bvk{{\bar{v}_\kappa}}
\def\bq{{\bar{q}}}

\def\bw{{\bar{w}}}

\def\tv{{\tilde{v}}}
\def\mJ{{\mathcal J}}
\def\bJ{\bar{\mathcal J}_\kappa}
\def\mI{{\mathcal I}}
\def\Op{{\Omega^+}}
\def\Om{{\Omega^-}}
\def\Gpm{{\Gamma}}
\def\Go{{\partial\Omega}}
\def\mP{{\mathcal P}}
\def\solnpspace{{H^{5.5}(\Op)}}

\def\solnspace{{H^{5.5}(\Omega)}}
\def\Ek{{E_\kappa}}

\def\mV{{\mathcal V}}
\def\sign{{\text{sign}}}
\def\E{{\mathcal E}}

\begin{document}
\title[Vortex sheets with surface tension]
{On the motion of vortex sheets with surface tension in the 3D
Euler equations with vorticity}

\author{C.H. Arthur Cheng}
\email{cchsiao@math.ucdavis.edu}
\author{Daniel Coutand}
\email{coutand@math.ucdavis.edu}
\author{Steve Shkoller}
\email{shkoller@math.ucdavis.edu}

\address{Department of Mathematics, University of California, Davis, CA 95616}
\date{May 1, 2007}

\maketitle

\section{Introduction}
The motion of vortex sheets with surface tension has been analyzed in the setting of
irrotational flows by Ambrose \cite{Am2003} and Ambrose \& Masmoudi \cite{AmMa2005} in 2D,
and by Ambrose \& Masmoudi \cite{AmMa2007} in 3D.   With irrotationality, the nonlinear
Euler equations reduce to the Laplace equation for the pressure function in the bulk, and
the motion of the vortex sheet is decoupled from that of the fluid, thus allowing
boundary integral methods to be employed.  In a general flow with vorticity, the full
two-phase Euler equations must be analyzed; in this situation, the vortex sheet is a
surface of discontinuity representing the material interface between two 
incompressible inviscid fluids with densities $\rho^+$ and $\rho^-$, respectively.  
The tangential velocity of the fluid suffers a jump
discontinuity along the material interface, leading to the well-known Kelvin-Helmholtz or 
Rayleigh-Taylor instabilities when surface tension is neglected.  The velocity of the 
vortex sheet is the normal component of the fluid velocity, whose continuity  across
the material interface $\Gamma(t)$ is enforced.  In addition to incompressibility, the continuity 
of the  normal component of velocity  across $\Gamma(t)$ is a fundamental difference between multi-D shock
wave evolution, wherein the velocity of the surface of discontinuity is determined by
the generalized Rankine-Hugoniot condition.  Nevertheless, the problems are mathematically
very similar, and we refer the reader to the book of Majda  \cite{Majda1984} for the
analysis of multi-D shocks. 

In the incompressible, {\it rotational} flow-setting, very little analysis has been made
of the two-phase Euler equations.  With surface tension
present, Shatah and Zeng \cite{ShZe2006} have obtained formal a priori estimates for 
smooth enough solutions, but the question of existence of smooth solutions remains open.
In this paper, following the methodology of Coutand \& Shkoller \cite{CoSh2007}, we prove
well-posedness for short-time for this problem.  

Let $\Op$ and $\Om$ denote two open bounded subsets of $\bbR^3$ such that
$\Omega = \Op \cup \Om$ denotes the total volume occupied by the two fluids,
and $\Gamma = \cls\Op \cap \cls\Om$ denotes the material interface.  We assume that
it is the region $\cls\Om$ that intersects $\partial \Omega$.

Let $\eta$ denote the Lagrangian flow map, satisfying
\begin{align*}
\eta_t(x,t) &= u(\eta(x,t),t) \qquad\forall\ x\in \Omega, t>0\,, \\
\eta(x,0)&=x \,.
\end{align*}
Let $\Op(t)$, $\Om(t)$ and $\Gpm(t)$ denote $\eta(t)(\Op)$,
$\eta(t)(\Om)$ and $\eta(t)(\Gpm)$, respectively, and let $u^\pm$ and
$p^\pm$ denote the velocity field and pressure function, respectively, 
in $\Omega^\pm(t)$. The incompressible Euler equations for the motion of two fluids
can be written as
\begin{subequations}\label{Euler}
\begin{alignat}{2}
\rho^\pm(u^\pm_t + \nabla_{u^\pm} u^\pm) + \nabla p^\pm &= 0 &&\qquad \text{in
\ \
$\Omega^\pm(t)$}\,, \\
\div u^\pm &= 0 &&\qquad \text{in
\ \
$\Omega^\pm(t)$}\,, \\
[p]_\pm &= \sigma H &&\qquad \text{on \ \ $\Gpm(t)$}\,,\\
[u\cdot n]_\pm &=0 &&\qquad \text{on \ \ $\Gpm(t)$}\,, \\
u^-\cdot n &= 0 &&\qquad\text{on \ \ $\bdy\Omega$}\,,\label{nonslipbc}\\
u(0) &= u_0 &&\qquad \text{on \ \ $\{t=0\}\times \Omega$}\,,
\end{alignat}
\end{subequations}
where the material interface $\Gamma(t)$ moves with speed $u(t)^+\cdot n(t)$, 
$\rho^+$ and $\rho^-$ are the densities of the two fluids
occupying $\Op(t)$ and $\Om(t)$, respectively,  $H(t)$ is twice the mean curvature of
$\Gpm(t)$, $\sigma >0$ is the surface tension parameter, 
and $n(t)$ denotes the outward-pointing unit normal on $\partial \Op(t)$.

\begin{theorem}[Main result]
Suppose that $\sigma>0$, $\Gamma$ is of class $H^4$, $\bdy\Omega$ is
of class $H^3$, and $u_0^\pm\in H^3(\Omega^\pm)$. Then, there exists
$T>0$, and a solution $(u^\pm(t),p^\pm(t),\Omega^\pm(t))$ of
(\ref{Euler}) with $u^\pm \in L^\infty(0,T;H^3(\Omega^\pm(t))$,
$p^\pm\in L^\infty(0,T;H^{2.5}(\Omega^\pm(t))$, and $\Gamma(t)\in
H^4$. The solution is unique if $u_0^\pm\in H^{4.5}(\Omega^\pm)$ and
$\Gamma\in H^{5.5}$.
\end{theorem}

The paper is organized as follows.  In Section \ref{sec::Notation}, we establish the
notation to be used throughout the paper.  In Section \ref{sec::Trace} we establish
low-regularity trace theorems of the normal and tangential components of $L^2$
vector fields with divergence and curl structure. In
Section \ref{linearkprob}, we introduce a regularized version of the Euler
equations (\ref{Euler});  the transport velocity and the domain are regularized using
the tool of horizontal convolution by layers that we introduced in \cite{CoSh2007}.
Additionally, a nonlinear parabolic regularization of the surface tension operator is
made in the Laplace-Young boundary condition (\ref{Eulerreg}d).
Section \ref{kprob} is devoted to the existence of solutions to (\ref{Eulerreg}).
In Section \ref{sec::time0}, we obtain estimates for the velocity, pressure, and their
time derivatives at time $t=0$.
Section \ref{sec::pressure} provides the pressure estimates that we need for a priori
estimates.
In Section \ref{kindpest}, we establish the $\kappa$-independent estimates for the
solutions of the $\kappa$-problem (\ref{Eulerreg}); this allows us to pass to the limit
as the regularization parameter $\kappa \to 0$ and prove existence of solutions to
(\ref{Euler}). In Section \ref{optimal}, we provide
the optimal regularity requirements on the data.  Finally, in Section \ref{sec::uniqueness}
we prove uniqueness of solutions.

\section{Notation} \label{sec::Notation}
Let ${\mathfrak n}:=\text{dim}(\Omega)=2$ or $3$.
We will use the notation $H^s(\Op)$ ($H^s(\Om)$) to denote either
$H^s(\Op;\bbR)$ ($H^s(\Om;\bbR)$) for a scalar function or
$H^s(\Op;\bbR^{\mathfrak n})$ ($H^s(\Om;\bbR^{\mathfrak n})$) for a
vector valued function, and we denote the 
$H^s(\Omega^\pm)$-norm by
\begin{align*}
\|w\|_{s,+} = \|w\|_{H^s(\Op)}\quad\text{and}\quad \|w\|_{s,-} =
\|w\|_{H^s(\Om)}\,.
\end{align*}
The $H^s(\Gpm)$- and $H^s(\Go)$-norms are denoted by
\begin{align*}
|w|_s = \|w\|_{H^s(\Gpm)}\quad\text{and}\quad |w|_{s,\Go} =
\|w\|_{H^s(\Go)}\,.
\end{align*}
For simplicity, we also use $\|w\|^2_{s,\pm}$ and $|w|^2_{s,\pm}$ to
denote $\|w^+\|^2_{s,+}+\|w^-\|^2_{s,-}$ and $|w^+|^2_s+|w^-|^2_s$,
respectively. That is,
\begin{align*}
\|w\|^2_{s,\pm} &= \|w^+\|^2_{s,+}+\|w^-\|^2_{s,-}\,, \\
|w|^2_{s,\pm} &= |w^+|^2_s+|w^-|^2_s\,.
\end{align*}
For $s\ge 1.5$, $\mV_+^s(T)$ denotes the space
\begin{align*}
\Big\{ w\in L^2(0,T;L^2(\Op))\ \Big|\ w\in L^2(0,T;H^s(\Op)\cap
L^\infty(0,T;H^{s-1.5}(\Op)) \Big\}
\end{align*}
with associated norm
\begin{align*}
\|w\|_{\mV_+^s(T)} = \sup_{t\in[0,T]} \|w(t)\|_{H^{s-1.5}(\Op)} +
\int_0^T \|w^+(s)\|_{H^s(\Op)}^2 ds\,,
\end{align*}
where $w$ can be either vector-valued or scalar-valued. The space
$\mV^s_-(T)$ is defined slightly differently, namely
\begin{align*}
\mV^s_-(T) \equiv \Big\{ w\in L^2(0,T;L^2(\Om))\ \Big|\ w\in
L^2(0,T;H^s(\Om)\cap L^\infty(0,T;H^{s-1}(\Om)) \Big\}
\end{align*}
with norm
\begin{align*}
\|w\|_{\mV_+^s(T)} = \sup_{t\in[0,T]} \|w(t)\|_{H^{s-1}(\Om)} +
\int_0^T \|w(s)\|_{H^s(\Om)}^2  ds\,.
\end{align*}
As in \cite{CoSh2007}, the energy function is defined as
\begin{align}
E_\kappa(t) =&\ \|\eta^+\|^2_{{\mathfrak n} + 2.5,+} +
\sum_{j=0}^{\mathfrak n} \|\partial_t^j v \|^2_{3.5-j,\pm} +
\|\partial_t^{{\mathfrak n} + 1}v\|^2_{0,\pm} +
\|\sqrt{\kappa}\eta^+\|^2_{{\mathfrak n} + 3.5,+} \label{energy} \\
& +\sum_{j=0}^{{\mathfrak n} + 1} \int_0^T
\|\sqrt{\kappa}\partial_t^j v^+\|^2_{4.5-j,+} dt \,. \nonumber
\end{align}

We use the notation $f^\pm = g^\pm + h^+ + k^-$ to
mean that 
\begin{align*}
f^+ = g^+ + h^+\,, \text{ and } f^- = g^- + k^- \,. 
\end{align*}

\section{Trace theorems} \label{sec::Trace}
The normal trace theorem which states that the existence of the
normal trace of a velocity field $w\in L^2(\Omega)$ relies on the
regularity of $\div u$ (see, for example, \cite{Temam}). If $\div w\in
H^1(\Omega)'$, then $w\cdot N$, the normal trace, exists in
$H^{-0.5}(\bdy\Omega)$ so that
\begin{align}
\|w\cdot N\|_{H^{-0.5}(\bdy\Omega)} \le C \Big[\|w\|^2_{L^2(\Omega)}
+ \|\div w\|^2_{H^1(\Omega)'}\Big] \label{normaltrace}
\end{align}
for some constant $C$ independent of $w$. In addition to the normal
trace theorem, we have the following.
\begin{theorem}
Let $w\in L^2(\Omega)$ so that $\curl w\in H^1(\Omega)'$, and let
$\tau_1$, $\tau_2$ be a basis of the vector field on $\bdy\Omega$,
i.e., any vector field $u$ can be uniquely written as $u^\alpha
\tau_\alpha$. Then
\begin{align}
\|w\cdot \tau_\alpha\|_{H^{-0.5}(\bdy\Omega)} \le C
\Big[\|w\|^2_{L^2(\Omega)} + \|\curl
w\|^2_{H^1(\Omega)'}\Big]\,,\qquad\alpha=1,2 \label{tangentialtrace}
\end{align}
for some constant $C$ independent of $w$.
\end{theorem}
\begin{proof} Given $\psi\in H^{0.5}(\bdy\Omega)$, let
$\phi_\alpha\in H^1(\Omega)$ be defined by
\begin{alignat*}{2}
\Delta\phi_\alpha &= 0 &&\quad\text{in $\Omega$}\,,\\
\phi_\alpha &= (N\times \tau_\alpha)\psi &&\quad\text{on
$\bdy\Omega$}\,.
\end{alignat*}
Then
\begin{align*}
\int_{\bdy\Omega} (w\cdot \tau_\alpha) \psi dS = \int_\Omega \curl w
\cdot \phi_\alpha dx - \int_\Omega \curl\phi_\alpha\cdot w dx
\end{align*}
and hence
\begin{align*}
\Big|\int_{\bdy\Omega} (w\cdot\tau_\alpha)\psi dS\Big| &\le C
\Big[\|w\|^2_{L^2(\Omega)} + \|\curl w\|_{H^1(\Omega)'}\Big]
\|\phi_\alpha\|_{H^1(\Omega)} \\
&\le C \Big[\|w\|^2_{L^2(\Omega)} + \|\curl w\|_{H^1(\Omega)'}\Big]
\|\psi\|_{H^{0.5}(\bdy\Omega)}
\end{align*}
which implies the desired inequality.
\end{proof}
Combining (\ref{normaltrace}) and (\ref{tangentialtrace}), we have
the following:
\begin{align}
\|w\|_{H^{-0.5}(\Go)} \le C\Big[\|w\|_{L^2(\Omega)} + \|\div
w\|_{H^1(\Omega)'} + \|\curl w\|_{H^1(\Omega)'}\Big]
\label{tracetemp}
\end{align}
for some constant $C$ independent of $w$.

\section{The regularized $\kappa$-problem}\label{linearkprob} 
Let $\Omega'$ be an open subset of $\Omega$ so that
$\Op\subset\subset \Omega' \subset\subset \Omega$. In the following
discussion, we will use $M^+ : \solnpspace\to\solnspace$ to denote a
fixed bounded extension operator (from the plus region to the whole
region) so that $M^+ v = 0$ in $\Omega'^c$ for all $v\in
\solnpspace$.

Let $v^+$ be the Lagrangian velocity in the plus region $\Op$, and
$v_e = M^+ v^+$ with $v_\kappa$ defined as the horizontal
convolution by layers of $v_e$. Let $\eta_\kappa = \id + \int_0^t
v_\kappa(s) ds$ be the Lagrangian coordinate (or flow map) of
$v_\kappa$, and the Jacobian $\mJ_\kappa$, the cofactor matrices
$a_\kappa$ and the normal $n_\kappa$ are defined accordingly.

The smoothed $\kappa$-problems is then defined as 
\begin{subequations}\label{Eulerreg}
\begin{alignat}{2}
\eta_e = \id + \int_0^t & v_e(s) ds &&\quad\text{in \ $[0,T]\times\Omega^\pm$}\,,\\
\rho^\pm \mJ_\kappa v^{\pm i}_t + (a_\kappa)_j^\ell (v^{-j} - v_e^{-
j}) v_{,\ell}^{\pm i} + (a_\kappa)^j_i q^\pm_{,j} &= 0
&&\quad\text{in \ $[0,T]\times\Omega^\pm$}\,,\\
(a_\kappa)^j_i v^{\pm i}_{,j} &= 0 &&\quad\text{in \ $[0,T]\times\Omega^\pm$}\,,\\
q^+ - q^- = - \sigma \Delta_g(\eta_e)\cdot n_\kappa - \kappa
&\Delta_0 (v^+ \cdot n_\kappa) &&\quad\text{on \ $[0,T]\times\Gpm$}\,,\\
v^+ \cdot n_\kappa &= v^- \cdot n_\kappa &&\quad\text{on \ $[0,T]\times\Gpm$}\,,\\
v^- \cdot n_\kappa &= 0 &&\quad\text{on \ $[0,T]\times\Go$}\,,\\
v^\pm(0) &= u^\pm_0 &&\quad\text{in \ $\{t=0\}\times\Omega^\pm$}\,,
\end{alignat}
\end{subequations}
where $g_\kappa = \eta_{\kappa,1}\cdot \eta_{\kappa,2}$ is the
induced metric on $\Gpm$. 
Note that since $M^+$ extends $v^+$ continuously to the whole domain
$\Omega$, $\eta_\kappa^+ = \eta_\kappa^-$ and $\eta_e^+ = \eta_e^-$
on $\Gpm$.

\begin{remark} \label{bcrmk}
Since $M^+ v = 0$ in $\Omega'^c$ for all $v\in\solnspace$, $\ba =
\id$ and $\nk = N$ on $\Go$. Therefore, the boundary condition
(\ref{Eulerreg}f) can also be written as $v^-\cdot N = 0$ where $N$
denote the outward pointing unit normal of $\Om$ on $\Go$.
\end{remark}

\section{Existence of solutions for the regularized
$\kappa$-problem} \label{kprob}
\subsection{The iteration between the solution in $\Op$ and $\Om$}
\label{extension} Let $(\bv^+, \bv^- \bq)\in \mV_+^{5.5}(T)
\times\mV_-^{4.5}(T) \times \mV_-^{3.5}(T)/\bbR$ be given, and let
$\bvk^+$ be the horizontal convolution by layers of $\bv^+$. Define
$\bv_e = M^+ \bvk^+$, the extension of $\bvk^+$, with the associated
Lagrangian map $\bek = \id + \int_0^t \bv_e(s) ds$ and cofactor
matrix $\ba = \bJ (\nabla\bek)^{-1}$ where $\bJ = \det(\nabla\bek)$
is the Jacobian. The normal vector $\bn$ is then defined by
\begin{align*}
\bn^i = \bg^{-\frac{1}{2}} \varepsilon_{ijk} \be^j_{\kappa,1}
\be^k_{\kappa,2} = \bg^{-\frac{1}{2}} (\ba)_i^j N_j\,.
\end{align*}

The process of finding solutions to (\ref{Eulerreg}) consists of
finding solutions to the following two problems. First, in the plus
region $\Op$, we solve
\begin{subequations}\label{Eulerplus}
\begin{alignat}{2}
\rho^+ \bJ w^i_t + (\ba)_i^j r_{,j} &= 0 && \quad \text{in \ \ $[0,T]\times\Op$}\,, \\
(\ba)^j_i w^i_{,j} &= 0 && \quad \text{in \ \ $[0,T]\times\Op$}\,, \\
r = \bq
- \sigma L_\bg(\be)\cdot\bn - \kappa&\Delta_{\bar{0}}
(w\cdot\bn) && \quad\text{on \ \ $[0,T]\times\Gpm$}\,,\\
w(0) &= u_0^+ && \quad \text{on \ \ $\{t=0\}\times\Omega$}\,,
\end{alignat}
\end{subequations}
where $w = u^+\circ\bek$, $r = p^+\circ\bek$, $\be = \id + \int_0^t
\bv(s)ds$ and $\Delta_{\bar{0}} = \bg^{-\frac{1}{2}}
\partial_\alpha[\sqrt{g_0} g_0^{\alpha\beta} \partial_\beta]$. Once
the solution $(w,r)$ to (\ref{Eulerplus}) is
obtained, 
then in the minus region $\Om$, we solve
\begin{subequations}\label{Eulerminus}
\begin{alignat}{2}
\rho^- \Big[\bJ v^i_t + (\ba)_j^\ell(\bv^- - \bv_e^-) v^i_{,\ell}
\Big]
+ (\ba)_i^j q_{,j} &= 0 && \quad \text{in \ \ $[0,T]\times\Om$}\,, \\
(\ba)^j_i v^i_{,j} &= 0 && \quad \text{in \ \ $[0,T]\times\Om$}\,, \\
v\cdot\bn &= w \cdot\bn
&& \quad \text{on \ \ $[0,T]\times\Gpm$}\,, \\
v\cdot\bn &= 0 && \quad \text{on \ \ $[0,T]\times\Go$}\,, \\
v(0) &= u_0^- && \quad \text{on \ \ $\{t=0\}\times\Omega$}\,,
\end{alignat}
\end{subequations}
where $v = u^-\circ\bek$, $q = p^-\circ\bek$.

This process introduces the map $\Phi:(\bv^+,\bv^-,\bq)\mapsto
(w,v,q)$, and the fixed-points of $\Phi$ provides solutions to
problem (\ref{Eulerreg}).

\subsection{Estimates for the solution in $\Op$} The
only difference between (\ref{Eulerplus}) and the one phase problem
studied in \cite{CoSh2007} is the presence of the term
$\bq$ 
in the boundary condition (\ref{Eulerplus}c). We note that if $\bq$
is smooth, then by exactly the same argument as in \cite{CoSh2007}, the
solution to (\ref{Eulerplus}) will be also be smooth, depending on
the regularity of the initial velocity $u_0^+$. Therefore, for $\bq$
given in $L^2(0,T;H^{3.5}(\Omega^-)/\bbR)$, we replace
(\ref{Eulerplus}c) by
\begin{align}
r = \bq_\epsilon
- \sigma L_\bg(\be)\cdot\bn - \kappa\Delta_{\bar{0}} (w\cdot\bn) \
\quad \text{on \ \ $\Gpm$}\,, \label{modifiedbc}
\end{align}
where $\bq_\epsilon$ denotes the horizontal convolution by layers of
$\bq$. The solution $w^\epsilon$ and $r^\epsilon$ to
(\ref{Eulerplus}a), (\ref{Eulerplus}b), (\ref{modifiedbc}) and
(\ref{Eulerplus}d) are smooth functions satisfying
\begin{align}
\|w^\epsilon\|^2_{0,+} + 
\int_0^t \Big[\|r^\epsilon\|^2_{3.5,+} + \kappa
|w^\epsilon\cdot\bn|^2_{5,\pm}\Big] ds &\le N(u_0) +
C(\kappa,\bv^+,\bq)\sqrt{t} \,, \label{energyest}
\end{align}
where $C(\kappa,\bv^+,\bq)$ denotes a constant depending on
$\kappa$, 
$\|\bv^+\|_{\mV_+^{5.5}(T)}$,
$\|\bq\|_{\mV_-^{3.5}(T)}$. Note that although this constant depends
on $\rho^+$ as well, we omit this dependence in the estimate since
it is a constant.

The divergence and curl estimates as in \cite{CoSh2007} can also be
carried on so that
\begin{align}
\|\curl w^\epsilon\|^2_{4.5,+} + \|\div w^\epsilon\|^2_{4.5,+} \le
N(u_0) + C(\kappa,\bv^+) \int_0^t \|w^\epsilon\|^2_{5.5,+} ds
\label{divcurl}
\end{align}
for some constant $C(\kappa,\bv^+)$ independent of the smooth
parameter $\epsilon$. Estimates (\ref{divcurl}) and
(\ref{energyest}) imply that
\begin{align*}
\int_0^t \|w^\epsilon(s)\|^2_{5.5,+} ds \le \frac{t}{\kappa} N(u_0)
+ C(\kappa,\bv^+) \int_0^t \int_0^s \|w^\epsilon(s')\|^2_{5.5,+} ds'
ds
\end{align*}
and the Gronwall inequality implies that
\begin{align}
\int_0^t \|w^\epsilon(s)\|^2_{5.5,+} ds \le
C(\kappa,u_0^+,\bv^+,\bq)\sqrt{t} \,.\label{epsilonest1}
\end{align}
By studying the elliptic problem for $r^\epsilon$ with the Dirichlet
boundary condition (\ref{modifiedbc}), we find that
\begin{align}
\int_0^t \|r^\epsilon(s)\|^2_{3.5,+} ds &\le
C(\kappa,u_0^+,\bv^+,\bq)\sqrt{t} \,. \label{epsilonest2}
\end{align}
(\ref{epsilonest2}) implies that $w^\epsilon_t\in
L^2(0,T;H^{2.5}(\Op))$ and by interpolations,
\begin{align}
\sup_{t\in[0,T]} \|w^\epsilon(t)\|^2_{4,+} \le \|u_0^+\|^2_{4,+} +
C(\kappa,u_0^+,\bv^+,\bq)\sqrt{T} \,. \label{epsilonest3}
\end{align}
(\ref{epsilonest3}) further implies that
\begin{align}
\|r^\epsilon\|^2_{2,+} \le N(u_0) +
C(\kappa,\delta,u_0^+,\bv^+,\bq)\sqrt{t} + \delta\|\bq\|^2_{2,-}\,.
\label{rest}
\end{align}
It also follows from (\ref{Eulerplus}a) that
$\|w^\epsilon_t\|^2_{\mV^{2.5}_+(T)}$ shares the same bound as
$\|r^\epsilon\|^2_{\mV^{3.5}_-(T)}$, i.e.,
\begin{align}
\|w^\epsilon_t\|^2_{\mV^{2.5}_+(T)} \le N(u_0) +
C(\kappa,\delta,u_0^+,\bv^+,\bq)\sqrt{t} + \delta\|\bq\|^2_{2,-}\,.
\label{wtest}
\end{align}
These $\epsilon$ independent estimates enable us to pass
$\epsilon\to 0$ and obtained solution $(w,r)$ to problem
(\ref{Eulerplus}) with estimate
\begin{align}
\|w\|^2_{\mV^{5.5}_+(T)} + \|w_t\|^2_{\mV^{2.5}_+(T)} +
\|r\|^2_{\mV^{3.5}_+(T)} \le N(u_0) + C_{\kappa,\delta} \sqrt{T} +
\delta\|\bq\|^2_{2,-} \,, \label{mainestimate1}
\end{align}
where $C_{\kappa,\delta}$ is the short hand notation for
$C(\kappa,\delta,u_0^+,\bv^+,\bq)$.

\subsection{Estimates for the solution in $\Om$} \label{Eulerm}
We will set up a iterative scheme in order to show the existence of
a solution to problem (\ref{Eulerminus}). Let $\bA^i_j = \bJ^{-1}
(\ba)^i_j$. For a given $\bw\in \mV^{4.5}_-(T)$ with
$\bw_t\in\mV^{2.5}_-(T)$, we solve first
\begin{subequations}\label{qminuseq}
\begin{alignat}{2}
\bA^k_i [\bA^j_i q_{,j}]_{,k} =& - \rho^- \bA^k_r \bv_{\kappa,s}^r
\bA^s_i \bw^i_{,k} - \rho^- \bA^k_i [\bA^\ell_j(\bv^{-j} -
\bv_e^{-j})
\bw_{,\ell}^i]_{,k} &&\text{ \ in \ $\Om$}\,,\\
\bA^j_i q_{,j} \bn^i = - \rho^-&[w_t\cdot \bn + (w - \bw)\cdot
\bar{n}_{\kappa t} + \bA^\ell_j (\bv^{-j} - \bv_e^{-j})
\bw^i_{,\ell} \bn^i] &&\text{ \ on \ $\Gpm$}\,,\\
\bA^j_i q_{,j} \bn^i = \rho^-&[ \bw \cdot \bar{n}_{\kappa t} -
\bA^\ell_j (\bvk^{-j} - \bv_e^{-j}) \bw^i_{,\ell} \bn^i] &&\text{ \
on \ $\Go$}\,.
\end{alignat}
\end{subequations}
Once a solution to (\ref{qminuseq}) is obtained, use this solution
$q$ in (\ref{Eulerminus}a) and solve the transport equation
\begin{align*}
\rho^- \Big[\bJ v^i_t + (\ba)_j^\ell(\bv^{-j} - \bv_e^{-j})
v^i_{,\ell} \Big]
+ (\ba)_i^j q_{,j} &= 0 && \quad \text{in \ \ $\Om$}\,, \\
v(0) &= u_0^- && \quad \text{in \ \ $\Om$}\,.
\end{align*}
Suppose we can prove that $v$ is actually in the space we start
with, then a fixed-point of the map $\Psi:\bw\mapsto v$ provides a
solution to problem (\ref{Eulerminus}).

We note that in this iterative scheme $\bA$ is always fixed with
estimates
\begin{align}
\|\id - \bA(t)\|_{4.5,-} \le C(\bv^+) \sqrt{t} \label{Aest}
\end{align}
for some constant $C$ depending on $\|\bv^+\|^2_{\mV^{5.5}_+(T)}$
but independent of $\kappa$. Therefore, by assuming that $T$ is
small enough (so that $C(\bv^+)T$ is small), it follows from
elliptic theory (see \cite{Eb2002}) that
\begin{subequations}
\begin{align*}
\|q\|^2_{3.5,-} &\le C\Big[\|\bw\|^2_{3.5,-} + \|w_t\|^2_{2.5,+} +
\|w\|^2_{2.5,+}\Big]\,,\\
\|q\|^2_{2.5,-} &\le C\Big[\|\bw\|^2_{2.5,-} + \|w_t\|^2_{1.5,+} +
\|w\|^2_{1.5,+}\Big]\,.
\end{align*}
\end{subequations}
Combining these two estimates and (\ref{mainestimate1}), by
interpolations we find that
\begin{align}
\|q\|^2_{\mV^{3.5}_-(T)} \le N(u_0) + C_{\kappa,\delta} \sqrt{T} +
\delta\|\bq\|^2_{2,-} + C T \Big[\|\bw\|^2_{\mV^{4.5}_-(T)} +
\|\bw_t\|^2_{\mV^{2.5}_-(T)}\Big] \,. \label{qminusest}
\end{align}
For the regularity of $v$, we mimic the divergence and curl
estimates as in \cite{CoSh2007}. In $\Om$,
\begin{align}
(\varepsilon_{ijk} \bA_j^\ell v_{,\ell}^{-k})_t + \bA_r^s (\bv^{-r}
- \bv_e^{-r}) (\varepsilon_{ijk} \bA_j^\ell v_{,\ell}^{-k})_{,s} =
B^i(v) \label{voreq}
\end{align}
where
\begin{align*}
B^i(v) &= \varepsilon_{ijk} \Big[\bA^\ell_r \bv^{-r}_{\kappa,s}
\bA^s_j v_{,\ell}^{-k} + \bA_r^s(\bv^{-r} - \bv_e^{-r})
\bA_{j,s}^\ell v_{,\ell}^{-k} - \bA_j^\ell [\bA_r^s(\bv^{-r} -
\bv_e^{-r})]_{,\ell} v_{,s}^{-k}\Big] \\
&= \varepsilon_{ijk} \Big[\bA^\ell_r \bv^{-r}_{\kappa,s} \bA^s_j
v_{,\ell}^{-k} - \bA_j^\ell \bA_r^s(\bv_{,\ell}^{-r} -
\bv_{e,\ell}^{-r}) v_{,s}^{-k}\Big]\,,
\end{align*}
a function of $\nabla v$, $\nabla \bv$ and $\nabla \be$, where we
use the identity $\bA^s_r \bA_{j,s}^\ell = \bA^s_j \bA_{r,s}^\ell$.
Let $\bar{\zeta}$ be the solution to
\begin{align*}
\bar{\zeta}_t^i = [\bA_j^i (\bv^{-j} -
\bv_e^{-j})]\circ\bar{\zeta}\,,
\end{align*}
i.e., $\bar{\zeta}$ is the flow map of the velocity field $\bA^T
(\bv - \bv_e)$, then
\begin{align}
\varepsilon_{ijk} \bA_j^\ell v_{,\ell}^{-k} = \Big[\curl u_0 +
\int_0^\cdot B^i(v)\circ\bar{\zeta} ds\Big]\circ\bar{\zeta}^{-1}\,.
\label{vorticitytemp}
\end{align}
Since
\begin{align*}
\Big[\int_0^t K(\bar{\zeta}(y,s),s)ds\Big]\circ\bar{\zeta}^{-1}(x,t)
= \int_0^t K(\bar{\zeta}(x,s-t),s)ds\,,
\end{align*}
(\ref{vorticitytemp}) implies
\begin{align}
\varepsilon_{ijk} \bA_j^\ell v_{,\ell}^{-k}(x,t) = (\curl
u_0)\circ\bar{\zeta}^{-1}(x,t) + \int_0^t
B^i(v)\circ\bar{\zeta}(x,s-t) ds\,. \label{vorticity}
\end{align}
We use (\ref{vorticity}) as the fundamental equality to proceed to
vorticity estimates in $\Om$. Since $\|\bar{\zeta}(t)\|^2_{4.5,-}
\le M_0 + C T \|\bv^-\|^2_{\mV^{4.5}_-(T)} \equiv C(\bv^-)$,
(\ref{vorticity}) implies that
\begin{align}
\|\curl_{\bek} v\|^2_{3.5,-} \le C(\bv^-) \Big[N(u_0) + \int_0^t
\|v\|^2_{4.5,-} ds\Big]\,. \label{curlvm}
\end{align}
Transforming back to the domain $\bek(\Om)$, we find that
\begin{align*}
\|\curl u\|^2_{H^{3.5}(\bek(\Om))} \le C(\bv^-)\Big[N(u_0) +
\int_0^t
\|u\|^2_{H^{3.5}(\bek(\Om))} ds\Big]\,. 
\end{align*}
We remark here that the restriction of obtaining higher regularity
is mainly due to the presence of $\nabla\bA$ in $B(v)$ that comes
from the transport term. Boundary conditions (\ref{Eulerminus}c) and
(\ref{Eulerminus}d) imply
\begin{align*}
\|u\cdot N\|^2_{H^4(\partial\bek(\Om))} = \|u\cdot
N\|^2_{H^4(\bek(\Gpm))} + \|u\cdot N\|^2_{H^4(\bek(\Go))} &\le
C\|w\|^2_{4.5,+} \,. 
\end{align*}
These two estimates and the divergence free constraint $\div u = 0$
lead to
\begin{align*}
X(T) \le C \|w\|^2_{\mV^{4.5}_+(T)} + C(\bv^-) \Big[T N(u_0) +
\int_0^T X(t) dt\Big]\,,
\end{align*}
where $\displaystyle{X(T) = \int_0^T \|u\|^2_{H^{4.5}(\bek(\Om))}
dt}$. Therefore, the Gronwall inequality implies
\begin{align*}
\int_0^t \|u\|^2_{H^{4.5}(\bek(\Om))} ds \le [1+C(\bv^-)T]N(u_0) +
C_{\kappa,\delta}\sqrt{T} + \delta \|\bq\|^2_{2,-}
\end{align*}
or equivalently,
\begin{align*}
\int_0^T \|v\|^2_{4.5,-} dt \le [1+C(\bv^-)T]N(u_0) +
C_{\kappa,\delta}\sqrt{T} + \delta \|\bq\|^2_{2,-}\,.
\end{align*}
For $T$ even smaller (so that $C(\bv^-)T$ is small), it follows from
(\ref{Eulerminus}a) that
\begin{align*}
& \int_0^T \|v_t\|^2_{2.5,-} dt \le C \int_0^T \Big[\|v\|^2_{3.5,-}
+ \|q\|^2_{3.5,-}\Big] ds \\
\le&\ N(u_0) + C_{\kappa,\delta}\sqrt{T} + \delta \|\bq\|^2_{2,-} +
C T\Big[\|\bw\|^2_{\mV^{4.5}_-(T)} +
\|\bw_t\|^2_{\mV^{2.5}_-(T)}\Big]\,.
\end{align*}
Therefore,
\begin{align}
&\|v\|^2_{\mV^{4.5}_-(T)} + \|v_t\|^2_{\mV^{2.5}_-(T)} +
\|q\|^2_{\mV^{3.5}_-(T)} \nonumber\\
\le&\ N(u_0) + C_{\kappa,\delta}\sqrt{T} + \delta \|\bq\|^2_{2,-} +
C T\Big[\|\bw\|^2_{\mV^{3.5}_-(T)} +
\|\bw_t\|^2_{\mV^{2.5}_-(T)}\Big]\,. \label{mainestimate2temp}
\end{align}
In the following sections, we will always assume that the initial
input $\bq$ satisfies $\|\bq\|^2_{\mV^{3.5}_-(T)} \le N(u_0)+1$. We
can choose a fixed but positive $\displaystyle{\delta< \frac{1}{2(
N(u_0)+1)}}$, and let $L$ be the collection of those elements $v\in
L^2(0,T;H^{4.5}(\Om))$ so that
\begin{align*}
\|v\|^2_{\mV^{4.5}_-(T)} + \|v_t\|^2_{\mV^{2.5}_-(T)} \le N(u_0) +
1\,.
\end{align*}
For a fixed $\kappa>0$, we choose $T$ small enough so that
\begin{align*}
C_{\kappa,\delta} \sqrt{T} + C T \Big[N(u_0) + 1\Big] \le
\frac{1}{2}\,.
\end{align*}
Clearly the map $\Psi$ maps from $L$ into $L$. Similar to the proof
in Section \ref{weakcont}, $\Psi$ can be shown to be weakly
continuous in $L^2(0,T_\kappa;H^{5.5}(\Om))$. Since $L$ is a closed
convex set in $L^2(0,T;H^{4.5}(\Om))$), by Tychonoff fixed-point
theorem, there is a fixed-point $v$ of the map $\Psi$ which provides
a solution to (\ref{Eulerminus}). Uniqueness follows from the fact
that (\ref{Eulerminus}) is linear.

\begin{remark}
It follows from (\ref{mainestimate2temp}) that
\begin{align}
\|v\|^2_{\mV^{4.5}_-(T)} + \|v_t\|^2_{\mV^{2.5}_-(T)} +
\|q\|^2_{\mV^{3.5}_-(T)} \le N(u_0) + 1 \,. \label{mainestimate2}
\end{align}
\end{remark}

\subsection{Weak continuity of the map $\Phi$}\label{weakcont} Let
$(\bv_m^\pm,\bq_m)$ converges weakly to $(\bv^\pm,\bq)$ in the space
$L^2(0,T;H^{5.5}(\Omega^\pm)) \times
L^2(0,T;H^{3.5}(\Omega^-)/\bbR)$, $\Phi(\bv_m^+,\bv_m^-,\bq_m) =
(w_m,v_m,q_m)$ and $\Phi(\bv^+,\bv^-,\bq) = (w,v,q)$.
Suppose that $\mJ_{\kappa m}^\pm$, $\bar{a}_{\kappa m}^\pm$,
$\bar{n}_{\kappa m}^\pm$ are constructed from $\bvk^\pm$
accordingly. By the property of convolution by layers and the weak
convergence, we have $(\mJ_{\kappa m}^\pm, \bar{a}_{\kappa m}^\pm,
\bar{n}_{\kappa m}^\pm)$ converges to $(\bJ^\pm, \ba^\pm, \bn^\pm)$
strongly in $[L^\infty(0,T;H^{4.5}(\Omega^\pm))]^3$. Since
$(w_m,r_m)$ satisfies
\begin{align*}
& \int_\Op \rho^+ \mJ_{\kappa m}^+ w_{m t}^i \varphi^i dx - \int_\Op
r_m (\bar{a}_{\kappa m}^+)_i^j \varphi_{,j}^i dx + \kappa \int_\Gpm
g_0^{\alpha\beta} (w_m\cdot\bar{n}_{\kappa m}^+)_{,\alpha} (\varphi\cdot\bar{n}_{\kappa m}^+)_{,\beta} dS \\
=&\ \sigma \int_\Gpm \Big[\bq_m + L_{\bar{g}_{\kappa
m}}(\be_m^+)\cdot\bar{n}_{\kappa m}^+\Big] (\bar{a}_{\kappa
m}^+)_i^j N_j \varphi^i dS \qquad\forall \varphi\in
H^{\frac{3}{2}}(\Op)\,, \nonumber
\end{align*}
and $(w_m,r_m)$ are uniformly bounded in
$L^2(0,T;H^{5.5}(\Op))\times L^2(0,T;H^{3.5}(\Op))$, it follows that
there exists $(\tilde{w},\tilde{r})$ so that
\begin{align*}
& \int_\Op \rho^+ \bJ^+ \tilde{w}_t^i \varphi^i dx - \int_\Op
\tilde{r} (\ba^+)_i^j \varphi_{,j}^i dx + \kappa \int_\Gpm
g_0^{\alpha\beta} (\tilde{w}\cdot\bn^+)_{,\alpha} (\varphi\cdot\bn^+)_{,\beta} dS \\
=&\ \sigma \int_\Gpm \Big[\bq + L_{\bg}(\be^+)\cdot\bn^+\Big]
(\ba^+)_i^j N_j \varphi^i dS \qquad\forall \varphi\in
H^{\frac{3}{2}}(\Op)\,. \nonumber
\end{align*}
By the uniqueness of the solution to the linearized problem,
$\tilde{w} = w$. Similar argument shows that the solution
$(v_m,q_m)$ to problem (\ref{Eulerminus}) with all the fixed
coefficients constructed from $\bv_m$ converges weakly to $(v,q)$,
the solution to problem (\ref{Eulerminus}). Therefore, the weak
continuity of the map $\Phi$ is established. 

\subsection{The fixed-point argument} The only thing we need to
check is that if there is $T>0$ and a closed convex set $K\subseteq
L^2(0,T;H^{5.5}(\Op))\times L^2(0,T;H^{4.5}(\Om)/\bbR)\times
L^2(0,T;H^{3.5}(\Om))$ so that $\Phi$ maps from $K$ into $K$. Let
$K$ be defined as the collection of those elements $(w^\pm,q)\in
L^2(0,T;H^{5.5}(\Op))\times L^2(0,T;H^{4.5}(\Om))\times
L^2(0,T;H^{3.5}(\Om)/\bbR)$ so that
\begin{align*}
\|w^+\|^2_{\mV^{5.5}_+(T)} + \|w_t^+\|^2_{\mV^{2.5}_+(T)} &\le
N(u_0)+ 1\,, \\
\|w^-\|^2_{\mV^{4.5}_-(T)} + \|w_t^-\|^2_{\mV^{2.5}_-(T)} +
\|q\|^2_{\mV^{3.5}_-(T)} &\le N(u_0) + 1\,.
\end{align*}
Recall that $\delta$ is fixed from the previous section. Similar to
the proof in the previous section, we choose $T>0$ small enough so
that
\begin{align*}
C_{\kappa,\delta}\sqrt{T} + T(N(u_0)+1) \le \frac{1}{2}\,.
\end{align*}
Then by estimates (\ref{mainestimate1}) and (\ref{mainestimate2}),
the map $\Phi$ indeed maps from $K$ into $K$. Therefore, the
Tychonoff fixed-point theorem implies the existence of a fixed-point
$(v,q)$ of $\Phi$.

\begin{remark}
This $T$ is $\kappa$-dependent.
\end{remark}

\begin{remark} Once a solution to problem (\ref{Eulerreg}) is
obtained, without loss of generality, we may assume that the
pressure and its time derivatives satisfy the
Poincar$\acute{\text{e}}$ inequality (\ref{Poincare}): let
$\displaystyle{\bar{q} = \frac{1}{|\Omega|} \Big(\int_\Op q^+ dx +
\int_\Om q^- dx\Big)}$. Since $q^+$ and $q^-$ is uniquely determined
up to the addition of a constant (constant in space), we can replace
$q^+$ and $q^-$ by $q^+ - \bar{q} (\equiv Q^+)$ and $q^- - \bar{q}
(\equiv Q^-)$
\begin{align}
\|Q\|^2_{0,\pm}  = \|Q\|^2_0 \le C \|\nabla Q \|^2_0 = C \|\nabla
Q\|^2_{0,\pm} \,. \label{Poincare}
\end{align}
\end{remark}

\subsection{Estimates of the divergence and curl of the velocity field}
\subsubsection{Divergence and curl estimates} In $\Op$, we can apply
exactly the same technique as in \cite{CoSh2007} to conclude the
following lemma.
\begin{lemma}[Divergence and curl estimates in $\Op$] \label{lemma1}
Let $L_1=\curl$ and $L_2=\div$, and let $\eta_0:=\eta(0)$ and
$$
M_0^+:= P(\|u_0^+\|_{2.5 + {\mathfrak n},+}, |\Gpm| _{3.5 +
{\mathfrak n}}, \sqrt{\kappa} \|u_0^+\|_{1.5+3{\mathfrak
n},+},\sqrt{\kappa} |\Gamma|_{1+3{\mathfrak n}})
$$
denote a polynomial function of its arguments. Then for $j=1,2$,
\begin{equation}\label{divcurlp}
\begin{array}{l}
\ \ \displaystyle{\sup_{t\in[0,T]} \|\sqrt{\kappa} L_j
\eta^+(t)\|^2_{2.5+{\mathfrak n},+} + \sum_{k=0}^{{\mathfrak n}+1}
\sup_{t\in[0,T]} \|L_j
\partial_t^k\eta^+(t)\|^2_{1.5+{\mathfrak n}-k,+} } \vspace{0.2cm}\\
\displaystyle{ + \sum_{k=0}^{{\mathfrak n}+1} \int_0^T
\|\sqrt{\kappa} L_j \partial_t^k v^+\|^2_{2.5+{\mathfrak n}-k,+} dt
\le M_0^+ + CT \mP(\sup_{t\in[0,T]} \Ek(t)) } \,.
\end{array}
\end{equation}
\end{lemma}
Similar to the way of obtaining (\ref{curlvm}), the following lemma
is valid as well.
\begin{lemma}[Divergence and curl estimates in $\Om$] \label{lemma2}
Let ${\mathfrak n}$, $L_1$ and $L_2$ be defined as those in Lemma
\ref{lemma1}, and
$$M_0^-:= P(\|u_0^-\|_{2.5 + {\mathfrak n},-}, |\Gpm| _{3.5 +
{\mathfrak n}}, \sqrt{\kappa} \|u_0^-\|_{1.5+3{\mathfrak
n},-},\sqrt{\kappa} |\Gamma|_{1+3{\mathfrak n}})\,.$$ Then for
$j=1,2$,
\begin{align}
\sum_{k=1}^{{\mathfrak n}+2} \sup_{t\in[0,T]} \|L_j \partial_t^k
v^-(t)\|^2_{1.5+{\mathfrak n}-k,-} \le M_0^- + CT
\mP(\sup_{t\in[0,T]} \Ek(t))\,. \label{divcurlm}
\end{align}
\end{lemma}

\subsubsection{$H^{-0.5}$-estimates for $v_{ttt}^\pm$ on the
boundary $\Gpm$ and $\Go$} By (\ref{Eulerreg}b),
\begin{alignat*}{2}
(\curl v^\pm_{ttt})^i =&\ \varepsilon_{ijk} \Big[ [\delta_j^\ell -
(\ak)_j^\ell] v_{ttt,\ell}^{\pm k} - (a_{\kappa tt})_j^\ell
v_{t,\ell}^{\pm k} - (a_{\kappa t})_j^\ell v_{tt,\ell}^{\pm k} &&\qquad\text{in $\Omega^\pm$}\,,\\
&\qquad - \Big((\ak)_j^\ell [A_r^s (v^{-r} - v_e^{-r}) v_{,s}^{-k}]_{,\ell}\Big)_{tt} \Big]\\
\div v^\pm_{ttt} =&\ (\delta_j^\ell - A_j^\ell) v_{ttt,\ell}^{\pm i}
- (A_{ttt})_i^j v_{,j}^{\pm i} - 3 (A_{tt})_i^j v_{t,j}^{\pm i} - 3
(A_t)_i^j v_{tt,j}^{\pm i} &&\qquad\text{in $\Omega^\pm$}\,.
\end{alignat*}
Since $v^\pm_{ttt}\in L^2(0,T; H^{1.5}(\Omega^\pm))$ (with $\kappa$
dependent estimate), $\curl v^\pm_{ttt}$ and $\div v^\pm_{ttt}$ are
both in $L^2(\Omega^\pm)$ and hence by (\ref{tracetemp}),
$\|v^\pm_{ttt}\|_{H^{-0.5}(\bdy\Omega^\pm)}$ exists. For $\varphi
\in H^1(\Om)$,
\begin{align*}
& \int_\Om \curl v^-_{ttt} \cdot \varphi dx \le \varepsilon_{ijk}
\int_{\bdy\Om} [\delta_j^\ell - (\ak)_j^\ell] v_{ttt}^k \varphi^i
N_\ell dS \\
&\qquad 
- \varepsilon_{ijk} \int_{\bdy\Om} A_j^\ell (\ak)_r^s (v^{-r} -
v_e^{-r}) v_{tt,\ell}^{-k} \varphi^i N_s dS +
C\mP(\Ek(t))\|\varphi\|_{1,-}\,.
\end{align*}
Since $\ak = \id$ and $v_e^-=0$ outside $\Omega'$, we find that
\begin{align*}
& \varepsilon_{ijk} \int_{\bdy\Om} A_j^\ell (\ak)_r^s (v^{-r} -
v_e^{-r}) v_{tt,\ell}^{-k} \varphi^i N_s dS \\
=&\ \varepsilon_{ijk} \Big[\int_\Gpm A_j^\ell \sqrt{\gk}
[n_\kappa\cdot (v_e^- - v^-)] v_{tt,\ell}^{-k} \varphi^i dS +
\int_\Go (v^-\cdot N) v_{tt,j}^{-k} \varphi^i dS\Big] \\
=&\ 0\,,
\end{align*}
where we use the boundary condition (\ref{Eulerreg}e) and
(\ref{Eulerreg}f) with $v_e^- = v^+$ on $\Gpm$ to conclude the last
equality. Therefore,
\begin{align*}
\Big|\int_\Om \curl v^-_{ttt} \cdot \varphi dx\Big| &\le
\varepsilon_{ijk} 
\int_{\Gpm} [\delta_j^\ell -
(\ak)_j^\ell] v_{ttt}^k \varphi^i dS 
+ C \mP(\Ek(t)) \|\varphi\|_{1,-} \\
&\le C \Big[t |v_{ttt}|_{-0.5,\pm} + \mP(\Ek(t))\Big]
\|\varphi\|_{1,-}
\end{align*}
which implies
\begin{align*}
\|\curl v^-_{ttt}\|_{H^1(\Om)'} \le C\Big[t|v^-_{ttt}|_{-0.5,\pm} +
\mP(\Ek(t))\Big]\,.
\end{align*}
Similarly,
\begin{align*}
\|\curl v^+_{ttt}\|_{H^1(\Op)'} &\le
C\Big[t |v^+_{ttt}|_{-0.5,\pm} + \mP(\Ek(t))\Big]\,,\\
\|\div v^\pm_{ttt}\|_{H^1(\Omega^\pm)'} &\le
C\Big[t\|v_{ttt}\|_{H^{-0.5}(\bdy\Omega^\pm)} + \mP(\Ek(t))\Big]\,.
\end{align*}
Therefore, by (\ref{tracetemp}),
\begin{align*}
\|v^\pm_{ttt}\|_{H^{-0.5}(\bdy\Omega^\pm)} &\le C
\Big[\|v^\pm_{ttt}\|_{L^2(\Omega)} + \|\div
v^\pm_{ttt}\|_{H^1(\Omega)'} + \|\curl v^\pm_{ttt}\|_{H^1(\Omega)'}\Big] \nonumber\\
&\le C T \|v_{ttt}^\pm\|_{H^{-0.5}(\bdy\Omega^\pm)} + C \mP(\Ek(t))
\,.
\end{align*}
It then follows from choosing $T>0$ small enough that
\begin{align}
|v_{ttt}|_{-0.5,\pm} + |v^-_{ttt}|_{-0.5,\Go} \le C \mP(\Ek(t)) \,.
\label{trace}
\end{align}

\section{Estimates for velocity, pressure, and their time derivatives at time $t=0$} 
\label{sec::time0}
In this section, we estimate the time derivatives of
the velocity and pressure at the initial time $t=0$. We use $w_k$,
$k=1,2,3$, and $q_\ell$, $\ell=0,1,2$, to denote $\partial_t^k v(0)$
and $\partial_t^\ell q(0)$.
Let $\varphi_\kappa$ be defined by
\begin{subequations}\label{phieq}
\begin{alignat}{2}
[\mJ^{-1} a_i^j (\mJ^{-1} a_i^k \varphi_\kappa)_{,k}]_{,j} &= 0 &&
\qquad\text{in \ \ $\Op$}\,,\\
\varphi_\kappa = - \sigma L_g(\eta_e)\cdot \nk &- \kappa \Delta_0
(v\cdot n_k) &&\qquad \text{on \ \ $\Gpm$}\,,\\
\varphi_\kappa &= 0 &&\qquad \text{on \ \ $\Go$}\,,
\end{alignat}
\end{subequations}
and the quantities $\varphi_0$, $\varphi_1$ and $\varphi_2$ be
defined by $\varphi_\kappa(0)$, $\varphi_{\kappa t}(0)$ and
$\varphi_{\kappa tt}(0)$, respectively.

Let $q_0^+$ and $q_0^-$ denote the initial pressure $q(0)$ in $\Op$
and $\Om$, respectively, then $q_0^+$ and $q_0^-$ satisfy
\begin{subequations}\label{initqeq}
\begin{alignat}{2}
- \frac{1}{\rho^+}\Delta (q_0^+ - \varphi_0) &= f^+ &&\qquad
\text{in \ \ $\Op$}\,,\\
- \frac{1}{\rho^-} \Delta q_0^- &= f^-
&&\qquad
\text{in \ \ $\Om$}\,,\\
\frac{1}{\rho^+}\frac{\partial q_0^+}{\partial N} &= - w_1^+ \cdot N
&& \qquad \text{on \ \ $\bdy\Op$}\,,\\
\frac{1}{\rho^-}\frac{\partial q_0^-}{\partial N} &= (- w_1^- +
\nabla_{(v_e^-(0) - u_0^-)} u_0^-) \cdot N && \qquad \text{on \ \
$\bdy\Om$}\,,
\end{alignat}
\end{subequations}
where $f^\pm = (\nabla u_0^\pm)^T:\nabla u_0^\pm$,
and $N$ denotes the unit normal of $\Gpm$ from $\Om$ into $\Op$, or
the outward unit normal of $\Go$.

\begin{remark}
The right-hand side of (\ref{initqeq}b) is in fact $f^- - (v_e^-(0)
- u_0^-)\cdot \nabla \div u_0^-$ while the last term is zero by the
divergence free constraint of the initial data.
\end{remark}
\vspace{.1in}

\noindent For all $\psi\in H^1(\Op)\cap H^1(\Om)$ so that $\psi^+ =
\psi^-$ on $\Gpm$, we have
\begin{align}
& \frac{1}{\rho^+} \int_\Op \nabla (q_0^+- \varphi_0) \cdot \nabla
\psi dx + \frac{1}{\rho^-} \int_\Om \nabla q_0^- \cdot \nabla \psi
dx = \int_\Op f^+ \psi dx
+ \int_\Om f^- \psi dx \nonumber \\
-& \int_\Go (w_1^- \cdot N - \nabla_{v_e^-(0) - u_0^-} u_0^- ) \psi
dS - \int_\Gpm (w_1^+ - w_1^- + \nabla \varphi_0)\cdot N \psi dS\,.
\label{weakformqtemp}
\end{align}
Since $[v\cdot \nk]_\pm = 0$ and $v^+\cdot \nk = 0$ on $\Go$, it
follows that
\begin{alignat*}{2}
w_1^+ \cdot N + u_0^+ \cdot n_{\kappa t}(0) &= w_1^- \cdot N + u_0^-
\cdot n_{\kappa t}(0) &&\qquad\text{on \ \
$\Gpm$}\,,\\
w_1^-\cdot N &= 0 && \qquad\text{on \ \ $\Go$}\,,
\end{alignat*}
and hence (\ref{weakformqtemp}) implies
\begin{align}
& \frac{1}{\rho^+} \int_\Op \nabla q_0^+ \cdot \nabla \psi dx +
\frac{1}{\rho^-} \int_\Om \nabla
(q_0^- - \varphi_0)\cdot \nabla \psi dx = \int_\Op f^+ \psi dx + \int_\Om f^- \psi dx \nonumber\\
+& \int_\Gpm [(u_0^+ - u_0^-)\cdot n_{\kappa t}(0) +\nabla \varphi_0
\cdot N] \psi dS +\int_\Go \nabla_{v_e^-(0) - u_0^-} u_0^- \cdot N
\psi dS \,. \label{weakforminitq}
\end{align}
Let $Q_0 = q_0^+ - \varphi_0$ in $\Op$ and $Q_0 = q_0^-$ in $\Om$.
Since $Q_0^+ = Q_0^-$ on $\Gpm$, we can use $Q_0$ (and its
difference quotients) as a test function in (\ref{weakforminitq}).
Since $n_{\kappa t}(0) = - g_0^{\alpha\beta} (u_{0
\kappa,\beta}\cdot N)\id_{,\alpha}$ and $\|v_e(0)\|_{k,-} \le C
\|u_0\|_{k,+}$, by standard difference quotient technique, for
$s>1.5$ for ${\mathfrak n} = 2$ or $s> 1.75$ for ${\mathfrak n}=3$,
\begin{align*}
&\|q_0^+\|^2_{s,+} + \|q_0^- - \varphi_0\|^2_{s,-} \\
\le&\ C \|f\|^2_{s-2,\pm} + C |u_0^+ \cdot \bar{n}_{\kappa t}(0) -
u_0^- \cdot \bar{n}_{\kappa t}(0) + \nabla_{v_e^-(0) - u_0^-}
u_0^-\cdot
N + \nabla\varphi_0\cdot N|^2_{s-1.5} \\
\le&\ C\mP(\|u_0\|^2_{s,\pm}) + C(|\Gpm|^2_{s+1.5} + \kappa
|u_0\cdot N|^2_{s+1.5})\,.
\end{align*}
By the elliptic estimate for $\varphi$ in (\ref{phieq}) together
with (\ref{Eulerreg}b), we find that for $s> 1 + s({\mathfrak n})$,
\begin{align}
\|w_1\|^2_{s-1,\pm} + \|q_0\|^2_{s,\pm} \le C \mP(\|u_0\|^2_{s,\pm},
|\Gpm|^2_{s+1.5} , \|\sqrt{\kappa} u_0^+\|^2_{s+2,+})\,.
\label{initestimate1}
\end{align}
For $j=1,2$, the
quantities $q_1^\pm$ and $q_2^\pm$ satisfy
\begin{alignat*}{2}
\frac{1}{\rho^\pm} \Delta (q_j^\pm - \varphi_j^+) &= h_j^\pm + k_j^-
+ \phi_j^+ && \qquad\text{in \ \ $\Omega^\pm$}\,,\\
\frac{1}{\rho^\pm}\frac{\partial q_j^\pm}{\partial N} &= -
(\partial_t^{j+1} v^\pm)(0)\cdot N + j (\nabla q_{j-1}^\pm)^T \nabla
u_{0 \kappa}^\pm N && \qquad\text{on \ \ $\bdy\Omega^\pm$}\,,\\
&\hspace{0.38cm} + 2 (j-1) \nabla q_0^\pm (\nabla u_{0 \kappa}^\pm
\nabla u_{0 \kappa}^\pm - \nabla w_{1 \kappa}^\pm) N + B_j^- &&
\end{alignat*}
where
\begin{align*}
h_1^\pm &= -2 u_{0 \kappa,i}^j w_{1,j}^{\pm i} + 2 u_{0 \kappa,i}^r
u_{0 \kappa,r}^\ell u_{0,\ell}^{\pm i} - w_{1 \kappa,i}^\ell
u_{0,\ell}^{\pm i} + \nabla q_0^\pm \cdot
\Delta u_{0 \kappa} + u_{0 \kappa,j}^i q^\pm_{0,ij}\,,\\
h_2^\pm &= -3 u_{0 \kappa,i}^j w_{2,j}^{\pm i} + 6 u_{0 \kappa,i}^r
u_{0 \kappa,r}^\ell w_{1,\ell}^{\pm i} - 3 w_{1 \kappa,i}^\ell
w_{1,\ell}^{\pm i} + 6 [\nabla u_{0
\kappa} \nabla u_{0 \kappa} \nabla u_{0 \kappa} ]^T:\nabla u^\pm_0 \\
&\hspace{.38cm} - 4 (\nabla u_0^\pm \nabla w_{1 \kappa}): (\nabla
u_{0 \kappa})^T - 2 (\nabla u_0^\pm \nabla u_{0
\kappa}): (\nabla w_{1 \kappa})^T + (\nabla w_{2 \kappa})^T:\nabla u^\pm_0 \\
&\hspace{0.38cm} + \div [2 (\nabla w_{1 \kappa})^T\nabla q_1^\pm +
2(\nabla u_{0 \kappa}\nabla u_{0 \kappa})^T\nabla q^\pm_0
- (\nabla w_{1 \kappa})^T\nabla q^\pm_0 ] \,,\\
k_1 &= \Big[- u_{0 \kappa,i}^k (v_e(0)^{-j} - u_0^{-j}) u_{0,k}^{-i}
+ (v_{e t}^{-j}(0) - w_1^{-j}) u_{0,j}^{-i} +
(v_e^{-j}(0) - u_0^-) w_{1,j}^{-i} \Big]_{,i}\,,\\
k_2 &= \Big[(2u_{0 \kappa,j}^\ell u_{0 \kappa ,\ell}^k - w_{1
\kappa,j}^k) (v_e^{-j}(0) - u_0^{-j})
u_{0,k}^{-i} + (v_{e tt}^{-j}(0) - w_2^{-j}) u_{0,j}^{-i} \\
&\hspace{.4cm} + (v_e^{-j}(0) - u_0^{-j}) w_{2,j}^{-i} - 2 u_{0
\kappa,j}^k (v_{et}^{-j}(0) - w_1^{-j}) w_{1,k}^{-i} -
2 u_{0 \kappa,j}^k (v_e^{-j}(0) - u_0^{-j}) w_{1,k}^{-i} \\
&\hspace{.4cm} + 2 (v_{et}^j(0) - w_1^{-j}) w_{1,j}^{-i} \Big]_{,i}
\,, \end{align*}
and
\begin{align*} \phi_1 &= - 2 \nabla u_{0
\kappa}: \nabla^2 \varphi_0 - \nabla
\varphi_0 \cdot \Delta u_{0 \kappa}\,,\\
\phi_2 &= - 4 (\nabla u_{0 \kappa} \nabla u_{0 \kappa}):\nabla^2
\varphi_0 + 2 \nabla w_{1 \kappa}: \nabla^2\varphi_0 - \div (\nabla
u_{0 \kappa}
\nabla u_{0 \kappa})\cdot \nabla \varphi_0 \\
&\hspace{0.38cm} + \Delta w_{1 \kappa}\cdot \nabla \varphi_0 + 2
\Delta u_{0 \kappa}\cdot \nabla \varphi_1 + 4 \nabla u_{0
\kappa}:\nabla^2 \varphi_1 + 2 \nabla [(\nabla u_{0
\kappa})^T \nabla\varphi_1]:\nabla u_{0 \kappa}\,,\\
B_1 &= (v_e^-(0) - u_0^-)^T \nabla u_{0 \kappa} (\nabla u_0^-)^T N -
(v_{et}^-(0) - w_1^-)^T \nabla u_0^- N
- (v_e^-(0) - u_0^-)^T \nabla w_1^- N\,,\\
B_2 &= \Big[\nabla u_0^- (2 \nabla u_{0 \kappa} \nabla u_{0 \kappa}
- \nabla w_{1 \kappa})(v_e^-(0) -
u_0^-) + \nabla u_0^- (v_{e tt}^-(0) - w_2^-) \\
&\hspace{0.4cm} + \nabla w_2^- (v_e^-(0) - u_0^-) - 2 (\nabla u_0^-
\nabla u_{0 \kappa})(v_{et}^-(0) - w_1^-) - 2 \nabla w_1^-
\nabla u_{0 \kappa} (v_e^-(0) - u_0^-) \\
&\hspace{0.4cm} - 2 \nabla w_1^- (v_{et}^-(0) - w_1^-) \Big]N \,,
\end{align*}
where $u_{0 \kappa}$, $w_{1\kappa}$ and $w_{2\kappa}$ are $M^+
u^+_{0 \kappa}$, $M^+ w^+_{1\kappa}$ and $M^+ w^+_{2\kappa}$,
respectively. Similar to the estimate of $q_0^\pm$, since on $\Gpm$,
\begin{align*}
q_j^+ - \varphi_j &= q_j^- \qquad\text{for $j=1, 2$}\,, \\
[v_{tt}^+(0) - v_{tt}^-(0)]\cdot N &= -2[w_1]_\pm \cdot
\bar{n}_{\kappa t}(0) - [u_0]_\pm \cdot \bar{n}_{\kappa tt}(0)\,,\\
[v_{ttt}^+(0) - v_{ttt}^-(0)]\cdot N &= -3 [w_2]_\pm \cdot
\bar{n}_{\kappa t}(0) - 3 [w_1]_\pm \cdot \bar{n}_{\kappa tt}(0) -
[u_0]_\pm \cdot \bar{n}_{\kappa ttt}(0)\,,
\end{align*}
and on $\Go$,
\begin{align*}
v_{tt}^-(0)\cdot N = v_{ttt}^-(0)\cdot N = 0\,,
\end{align*}
we find that for $s\ge 2$,
\begin{align*}
& \|w_2\|^2_{s-2,\pm} + \|q_1\|^2_{s-1,\pm} \nonumber\\
\le&\ C\mP(\|u_0\|^2_{s,\pm}, |\Gpm|^2_{s+1.5}, \|\sqrt{\kappa}
u_0^+\|^2_{s+2,+}, \|\sqrt{\kappa} w_1^+\|^2_{s+1,+})
\end{align*}
and for $s\ge 3$, 
\begin{align*}
& \|w_3\|^2_{s-3,\pm} + \|q_2\|^2_{s-2,\pm} \\ \le&\
C\mP(\|u_0\|^2_{s,\pm}, |\Gpm|^2_{s+1.5} + \kappa |u_0\cdot
N|^2_{s+1.5} + \kappa |w_1^+\cdot N|^2_{s+0.5} + \kappa |w_2^+\cdot
N|^2_{s-0.5})\,,\nonumber
\end{align*}
where we also use the boundedness of the extension operator $M$ so
that
\begin{align*}
\|\partial_t^k v_{e}(0)\|_{s,-} \le C\|w_k^+\|^2_{s,+}\,.
\end{align*}


\section{Pressure estimates} \label{sec::pressure}
The estimates for the pressure
and its time derivatives are exactly the same as (12.1) in
\cite{CoSh2007}. In \cite{CoSh2007}, the $L^2$-estimate for the pressure is
found by studying a Dirichlet problem, but in the two-phase problem
with fixed outer boundary, the $L^2$-estimate is not necessary
because of the Poincare inequality. Therefore,
\begin{align}
\|q(t)\|^2_{3.5,\pm} + \|q_t(t)\|^2_{2.5,\pm} +
\|q_{tt}(t)\|^2_{1,\pm} \le C \mP(\Ek(t)) \label{pressureest}
\end{align}
for some constant $C$ independent of $\kappa$.

\begin{remark}
The estimates for $q^-$, $q^-_t$ and $q^-_{tt}$ require the control
of $\|v^-\|_{3.5,-}$, $\|v^-_t\|_{2.5,-}$, $\|v^-_{tt}\|_{1.5,-}$
and $\|v^-_{ttt}\|_{0,-}$, respectively. This is the only reason we
need to include the estimates for $\partial_t^j v^-$ into our
definition of energy (\ref{energy}). Note that we do not need
$\|\eta^-\|_{4.5,-}$ in order to control $\partial_t^j q^-$.
\end{remark}

\section{$\kappa$-independent estimates}\label{kindpest} 

We also make use of the following
inequality which follows from Morrey's inequality (see (2.6) in
\cite{CoSh2007}). For $U\in W^{1,p}(\Gpm)$,
\begin{align}
|U_\kappa(x) - U(x)| \le C \kappa^{1-\frac{\mathfrak n}{p}} |DU|_p
\label{Morrey}
\end{align}

Test (\ref{Eulerreg})
against a function $\varphi\in H^{\frac{3}{2}}(\Op)\cap
H^{\frac{3}{2}}(\Om)$ with $\varphi^-\cdot N = 0$ on $\Go$,
\begin{align}
& \int_\Op \rho^+ \mJ_\kappa v_t^{+i} \varphi^i dx + \int_\Om
\rho^-\Big[\mJ_\kappa v^{-i}_t + (\ak)_j^k (v^{-j} -
v_e^{-j}) v^{-i}_{,k}\Big]\varphi^i dx \nonumber\\
+& \int_\Op (\ak)_i^j q^+_{,j}  \varphi^{+i} dx + \int_\Om (\ak)_i^j
q^-_{,j} \varphi^{-i} dx = 0\,. \label{weakform}
\end{align}
Similar to those estimates in \cite{CoSh2007}, the $\kappa$-independent
estimate consists of studying the three time differentiated problem,
three tangential space differentiated problem, and the intermediate
problems with mixing time and tangential space derivatives. Most of
the estimates are essentially the same as those in \cite{CoSh2007}, and
in the following sections we only list those terms which required
further study.

Before proceeding, we remark that those energy estimates in
\cite{CoSh2007} can be refined a bit further. For example, the energy
estimate for the third time-differentiated $\kappa$-problem ((12.6)
in \cite{CoSh2007}) can be refined as
\begin{align*}
&\sup_{t\in[0,T]} \Big[\|v_{ttt}\|^2_0 + |v_{tt}\cdot\nk|_1^2\Big] +
\int_0^T |\sqrt{\kappa} \partial_t^3 v\cdot\nk|^2_1 dt \le
M_0(\delta) + \delta \sup_{t\in[0,T]} \Ek(t) \\
+&\ C T \mP(\sup_{t\in[0,T]} \Ek(t)) + C(\delta)
\Big[\|v_t\|^2_{2.5} + \|v\|^2_{3.5} + \|\eta_e\|^2_{4.5} + \int_0^T
\|\sqrt{\kappa} v_{tt}\|^2_{2.5} dt \Big]\,,
\end{align*}
where the difference is not having
\begin{align*}
C \sup_{t\in[0,T]}\Big[P(\|v_t\|^2_{2.5}) + P(\|v\|^2_{3.5} +
P(\|\eta\|^2_{4.5})\Big] + CP(\|\sqrt{\kappa}
v_{tt}\|^2_{L^2(0,T;H^{2.5}(\Omega))})
\end{align*}
on the right-hand side of the inequality. To see this, for example,
one such term comes from estimating
\begin{align*}
\sup_{t\in[0,T]} |P(v,\partial\eta_\kappa)|_{L^\infty}(\Go) \int_0^T
|\sqrt{\kappa} \partial_t^3 v\cdot\nk|_1 |\sqrt{\kappa} \partial_t^2
v_\kappa|_2 dt\,.
\end{align*}
Since $P(v,\partial\eta_\kappa)_t \in L^\infty(0,T;L^1(\Gamma))$, by
the fundamental theorem of calculus,
\begin{align*}
\sup_{t\in[0,T]} |P(v,\partial\eta_\kappa)|_{L^\infty(\Gamma)} \le
M_0 + C T \mP(\sup_{t\in [0,T]} \Ek(t))
\end{align*}
and hence
\begin{align*}
&\sup_{t\in[0,T]} |P(v,\partial\eta_\kappa)|_{L^\infty} \int_0^T
|\sqrt{\kappa} \partial_t^3 v\cdot\nk|_1 |\sqrt{\kappa} \partial_t^2
v_\kappa|_2 dt \\
\le&\ \delta \sup_{t\in[0,T]} \Ek(t) + CT\mP(\sup_{t\in[0,T]}
\Ek(t)) + C(\delta) \int_0^T \|\sqrt{\kappa} v^+_{tt}\|^2_{2.5}
dt\,,
\end{align*}
instead of having $\|\sqrt{\kappa}
v_{tt}\|^4_{L^2(0,T;H^{2.5}(\Omega))}$ in the bound shown in
\cite{CoSh2007}. Therefore, the energy estimates we cite from
\cite{CoSh2007} will have only one polynomial type of term in the bound:
$CT\mP(\sup_{t\in[0,T]} \Ek(t))$.

In this section, we will make use of the following equality which
follows from (\ref{Eulerreg}e)
\begin{align}
\nk\cdot(v_{ttt}^+ - v_{ttt}^-) = n_{\kappa ttt}\cdot(v^- - v^+) - 3
n_{\kappa tt}\cdot (v_t^+ - v_t^-) - 3 n_{\kappa t}\cdot(v_{tt}^+ -
v_{tt}^-)\,. \label{vneq}
\end{align}

\subsection{Estimates for the third time-differentiated $\kappa$-problem} Three time
differentiate (\ref{weakform}), and then use $v_{ttt}$ as a test
function and integrate in time from $0$ to $T$, we find that
\begin{align*}
& \sum_{\sign = \pm} \int_0^T \int_{\Omega^\sign} \rho^\sign
(\mJ_\kappa v_t^{\sign i})_{ttt} v_{ttt}^{\sign i} + \Big[(\ak)^j_i
q^\sign_{,j} \Big]_{ttt} v_{ttt}^{\sign i} dx dt\\
+&\int_0^T \int_\Om \rho^- \Big[(\ak)^k_j (v^{-j} - v_e^{-j})
v_{,k}^{-i}\Big]_{ttt} v_{ttt}^{-i} dx dt = 0\,.
\end{align*}
The terms needed additional analysis are
\begin{align*}
\mI_1 &= \int_0^T \int_\Om \rho^- \Big[(\ak)^k_j (v^{-j} - v_e^{-j})
v_{,k}^{-i}\Big]_{ttt} v_{ttt}^{-i} dx dt\,,\\
\mI_2 &= \int_0^T \int_\Op \Big[(\ak)_i^j q^+_{,j} \Big]_{ttt}
v_{ttt}^{+i} dx dt + \int_0^T \int_\Om \Big[(\ak)_i^j
q^-_{,j}\Big]_{ttt} v_{ttt}^{-i} dx dt\,.
\end{align*}
The worst terms of $\mI_1$ is when all the time derivatives hit
$v_{,k}$, while the other combinations are bounded by $C\mP(\Ek)$.
Therefore,
\begin{align*}
\mI_1 \le& \int_0^T \int_\Om \rho^- (\ak)^k_j (v^{-j} - v_e^{-j})
v_{ttt,k}^{-i} v_{ttt}^{-i} dx dt + CT\mP(\Ek) \\
=&\ \frac{1}{2} \int_0^T \int_\Om \rho^- (\ak)^k_j (v^{-j} -
v_e^{-j}) |v_{ttt}|^2_{,k} dx dt + CT\mP(\Ek) \\
=
& - \frac{1}{2} \int_0^T \int_{\bdy\Om} \rho^- (\ak)^k_j (v^{-j} -
v_e^{-j}) N_k |v_{ttt}|^2 dS dt + CT\mP(\Ek)\,.
\end{align*}
The boundary of $\Om$ consists of $\Gpm$ and $\Go$. On $\Go$, $\ak =
\id$ and $v_e = 0$. Therefore, by (\ref{Eulerreg}f),
\begin{align*}
\int_\Go \rho^- (\ak)^k_j (v^{-j} - v_e^{-j}) N_k |v_{ttt}|^2 dS =
\int_\Go \rho^- (v^-\cdot N) |v_{ttt}|^2 dS = 0\,.
\end{align*}
On $\Gpm$, since $v_e^- = v^+$ and $n_\kappa^j =
g_\kappa^{-\frac{1}{2}} (\ak)_j^k N_k$, boundary condition
(\ref{Eulerreg}e) implies that
\begin{align*}
\int_\Gpm \rho^- (\ak)^k_j (v^{-j} - v_e^{-j}) N_k |v_{ttt}|^2 dS =
\int_\Gpm \sqrt{g_\kappa} \rho^- [v\cdot \nk]_\pm |v_{ttt}|^2 dS =
0\,.
\end{align*}
Therefore,
\begin{align}
\mI_1 \le  CT\mP(\Ek)\,. \label{I1}
\end{align}
The worst terms of $\mI_2$ is when all the time derivatives hit $q$.
Therefore,
\begin{align*}
\mI_2 \le& \int_0^T\int_\Op (\ak)_i^j q^+_{ttt,j} v_{ttt}^{+i} dx dt
+ \int_0^T \int_\Om (\ak)_i^j q^-_{ttt,j} v_{ttt}^{-i} dx dt + CT\mP(\Ek) \\
=& \sum_{\sign=\pm}\int_0^T \Big[\int_\Gpm q^\sign_{ttt} (\ak)^j_i
N_j v_{ttt}^{\sign i} dS - \int_{\Omega^\sign} q^\sign_{ttt}
(\ak)^j_i v_{ttt,j}^{\sign i} dx\Big] dt + CT\mP(\Ek) \\
=&\ \mI_{21} + \mI_{22} + CT\mP(\Ek)\,.
\end{align*}
For $\mI_{21}$, it follows that
\begin{align*}
\mI_{21} =& \int_0^T \int_\Gpm \sqrt{\gk} n_\kappa^i \Big[q_{ttt}^+
v_{ttt}^{+i} - q_{ttt}^- v_{ttt}^{-i}\Big] dS dt \\
=& \int_0^T \Big[\int_\Gpm \sqrt{\gk} q_{ttt}^- (v_{ttt}^+ -
v_{ttt}^-)\cdot\nk dS + \int_\Gpm \sqrt{\gk} (q^+ - q^-)_{ttt}
(v_{ttt}^+\cdot\nk) dS \Big] dt\,.
\end{align*}
By (\ref{vneq}) and substituting $-\sigma L_g(\eta_e)\cdot\nk -
\kappa \Delta_0 (v\cdot \nk) \nk$ for $(q^+ - q^-)$, we apply the
estimates as in \cite{CoSh2007} to obtain
\begin{alignat*}{2}
\mI_{21} \le& - \int_0^T \int_\Gpm \sqrt{\gk} q_{ttt}^-\Big[(v^+ -
v^-)\cdot n_{\kappa ttt} + 3 n_{\kappa t}\cdot(v_{tt}^+
- v_{tt}^-)\Big]dS dt &&\quad(\equiv I_{21_a})\\
& - 3 \int_0^T \int_\Gpm \sqrt{\gk} q_{ttt}^- (v_t^+ - v_t^-)\cdot
n_{\kappa tt} dS dt
&&\quad(\equiv I_{21_b}) \\
& + \delta \sup_{t\in[0,T]}\Ek(t) + M_0(\delta) + C T
\mP(\sup_{t\in[0,T]}\Ek(t)) + C(\delta) \Big[\|v^+_t\|^2_{2.5,+} \\
& + \|v^+\|^2_{3.5,+} + \|\eta_e\|^2_{4.5,+} \Big]\,.
\end{alignat*}
Integrating by parts in time, since $\Big[\sqrt{\gk} (v_t^+ -
v_t^-)\cdot n_{\kappa tt})\Big]_t\in L^\infty(0,T;L^2(\Gpm))$, using
the same techniques as in \cite{CoSh2007}, we find that
\begin{align}
\frac{\mI_{21_b}}{3} =&\ \int_0^T \int_\Gpm q_{tt}^- \Big[\sqrt{\gk}
(v_t^+ - v_t^-)\cdot n_{\kappa tt}\Big]_t dS dt - \int_\Gpm q_{tt}^-
\sqrt{\gk} (v_t^+ - v_t^-)\cdot n_{\kappa
tt} dS\Big|_{t=0}^{t=T} \nonumber\\
\le&\ \delta \sup_{t\in[0,T]}\Ek(t) + M_0(\delta) + C T
\mP(\sup_{t\in[0,T]}\Ek(t))\,. \label{I21b}
\end{align}
Let the first and the second term of $\mI_{21_a}$ be denoted by
$\mI_{21_{a_1}}$ and $\mI_{21_{a_2}}$, respectively. Integrating by
parts in time,
\begin{align*}
\frac{\mI_{21_{a_2}}}{3} =& \int_0^T \int_\Gpm
q_{tt}^-\Big[\sqrt{\gk} n_{\kappa t} \cdot (v_{ttt}^+ - v_{ttt}^-) +
(\sqrt{\gk} n_{\kappa t})_t \cdot
(v_{tt}^+ - v_{tt}^-)\Big] dS dt \\
& - \int_\Gpm \sqrt{\gk} q_{tt}^- n_{\kappa t} \cdot (v_{tt}^+ -
v_{tt}^-) dS \Big|_{t=0}^{t=T}\,.
\end{align*}
By (\ref{trace}), $\displaystyle{\Big[\sqrt{\gk} n_{\kappa t}\cdot
(v_{tt}^+ - v_{tt}^-)\Big]_t \in L^2(0,T;H^{-0.5}(\Gamma))}$. Since
$q_{tt}^- \in L^\infty(0,T;H^{0.5}(\Gamma))$, it follows that
\begin{align*}
\int_\Gpm \sqrt{\gk} q_{tt}^- n_{\kappa t}\cdot (v_{tt}^+ -
v_{tt}^-) dS \Big|_{t=0}^{t=T} \le \delta \sup_{t\in[0,T]}\Ek(t) +
M_0(\delta) + C T \mP(\sup_{t\in[0,T]}\Ek(t))\,.
\end{align*}
Again by (\ref{trace}), we can estimate the first integral of
$\mI_{21_{a_2}}$ and obtain
\begin{align}
\mI_{21_{a_2}} \le \delta \sup_{t\in[0,T]}\Ek(t) + M_0(\delta) + C T
\mP(\sup_{t\in[0,T]}\Ek(t))\,.
\end{align}
For $\mI_{21_{a_1}}$, integrating by parts in time again,
\begin{align*}
\mI_{21_{a_1}} =& \int_0^T \int_\Gpm q_{tt}^- \Big[\sqrt{\gk} (v^+ -
v^-)\cdot \partial_t^4\nk + [\sqrt{\gk}
(v^+ - v^-)]_t n_{\kappa ttt}\Big] dS dt \\
& - \int_\Gpm \sqrt{\gk} q_{tt}^- (v^+ - v^-)\cdot n_{\kappa ttt} dS
\Big|_{t=0}^{t=T}\,.
\end{align*}
The second term of $\mI_{21_{a_1}}$ can be bounded by $C
T\mP(\sup_{t\in[0,T]}\Ek(t))$ since the integrand is in
$L^\infty(0,T;L^1(\Gpm))$. Since $n_{\kappa ttt} \sim
F_1(\partial\eta_\kappa) \partial v_{\kappa tt} +
F_2(\partial\eta_\kappa,\partial v_\kappa)
\partial v_{\kappa t}$, by the fact that $\Big[\sqrt{\gk} (v^+ - v^-) (F_1
+ F_2(\partial\eta_\kappa,\partial v_\kappa)\partial v_{\kappa
t})\Big]_t\in L^\infty(0,T;L^2(\Gpm))$ and
$H^{0.5}(\Gamma)$-$H^{-0.5}(\Gamma)$ duality pairing,
\begin{align*}
&\int_\Gpm \sqrt{\gk} q_{tt}^- (v^+ - v^-)\cdot n_{\kappa ttt}
dS\Big|_{t=0}^{t=T} \\
\le&\ M_0 + \Big[M_0 + C T \mP(\sup_{t\in[0,T]}\Ek(t))\Big]
|q_{tt}^-(T)|_{0.5} \Big[|\partial v_{\kappa tt}(T)|_{-0.5} + 1\Big] \\
\le&\ M_0 + M_0 \|q_{tt}^-(T)\|_{1,-} \Big[\|v_{\kappa
tt}(T)\|_{1,+} + 1 \Big]
+ CT \mP(\sup_{t\in[0,T]}\Ek(t))\\
\le&\ M_0(\delta) + \delta \sup_{t\in[0,T]}\Ek(t) + C T
\mP(\sup_{t\in[0,T]}\Ek(t))\,,
\end{align*}
where $\|v_{\kappa tt}(T)\|_{0,+} \le M_0 + CT\mP(\sup_{t\in[0,T]}
\Ek(t))$ and Young's inequality are used to obtain the last
inequality.

It remains to estimate the first term of $\mI_{21_{a_1}}$ in order
to complete the estimate of $\mI_{21}$. We write the first term as
\begin{align*}
&- \int_0^T \int_\Gpm (q^+_{tt} - q^-_{tt}) \sqrt{\gk} (v^+ - v^-)
\cdot \partial_t^4\nk dS dt \qquad(\equiv \mI_3)\\
&+ \int_0^T \int_\Gpm q^+_{tt} \sqrt{\gk} (v^+ - v^-) \cdot
\partial_t^4\nk dS dt\,. \qquad(\equiv \mI_4)
\end{align*}
By $n_{\kappa t} = - g_\kappa^{\alpha\beta} (v_{\kappa,\alpha}\cdot
\nk) \eta_{\kappa,\beta}$,
\begin{align*}
\mI_4 =&\ -\int_0^T \int_\Gpm q_{tt}^+ \sqrt{\gk} (v^+ -
v^-)\cdot\eta_{\kappa,\beta}g_\kappa^{\alpha\beta} (v_{\kappa
ttt,\alpha}\cdot \nk) dS dt \\
& + \int_0^T \int_\Gpm q_{tt}^+ F(\partial\eta_\kappa,
\partial v_\kappa, \partial v_{\kappa t}) (\partial v_{\kappa tt}\cdot \nk + 1)dS
dt
\end{align*}
where the second integral is bounded by $CT\mP(\sup_{t\in [0,T]}
\Ek(t))$. For the first term, 
\begin{align*}
&\int_0^T \int_\Gpm q_{tt}^+ \sqrt{\gk} (v^+ -
v^-)\cdot\eta_{\kappa,\beta}g_\kappa^{\alpha\beta} (v_{\kappa
ttt,\alpha}\cdot \nk) dS dt \\
=&\int_0^T \int_\Gpm q_{tt}^+ \sqrt{\gk} (v^+ -
v^-)\cdot\eta_{\kappa,\beta}g_\kappa^{\alpha\beta} \Big[(v_{\kappa
ttt}\cdot \nk)_{,\alpha} - (v_{\kappa ttt} \cdot \nk_{,\alpha})\Big]
dS dt\,.
\end{align*}
It follows from $H^{0.5}(\Gpm)$-$H^{-0.5}(\Gpm)$ duality pairing and
(\ref{trace}) that the term with $(v_{\kappa ttt}\cdot
n_{\kappa,\alpha})$ is also bounded by $CT\mP(\sup_{t\in [0,T]}
\Ek(t))$.

Let $\xi$ be a non-negative cut-off function so that
$\supp\xi\subset \bigcup_i \supp\alpha_i$ and $\xi = 1$ on $\Gpm$.
Integrating by parts in space, since $\bdy\Gpm = \phi$, by the
divergence theorem,
\begin{align}
&\int_0^T\int_\Gpm q_{tt}^+ \sqrt{\gk} (v^+ - v^-)\cdot
\eta_{\kappa,\beta} g_\kappa^{\alpha\beta} (v_{\kappa
ttt}\cdot\nk)_{,\alpha} dS dt \nonumber\\
=&\ -\int_0^T \int_\Gpm \Big[\xi \sqrt{\gk} (v^+ - v^-)\cdot
\eta_{\kappa,\beta} g_\kappa^{\alpha\beta}\Big]_{,\alpha} q_{tt}^+
(v_{\kappa ttt}\cdot\nk) dS dt \quad(\equiv \mI_{41}) \nonumber\\
&\ - \int_0^T \int_\Gpm q_{tt,\alpha}^+ \xi (v^+ - v^-)\cdot
\eta_{\kappa,\beta} g_\kappa^{\alpha\beta} v_{\kappa ttt}^i
(\ak)_i^j N_j dS dt \nonumber\\
=&\ \mI_{41} - \int_0^T \int_\Op (\ak)_i^j \Big[\xi q_{tt,\alpha}^+
(v^+ - w)\cdot\eta_{\kappa,\beta} g_\kappa^{\alpha\beta} v_{\kappa
ttt}^i\Big]_{,j} dx dt\,, \label{bdytoint}
\end{align}
where $w$ is an $H^{5.5}$-extension of $v^-$ to $\Op$. By
(\ref{trace}), $\mI_{41}$ can be bounded by $CT\mP(\sup_{t\in [0,T]}
\Ek(t))$ as well. For the rest terms, there are two worst cases:
when the derivative $\partial_j$ hits $q_{tt,\alpha}^+$ or
$v_{\kappa ttt}^i$. For the latter case, by inequality
(\ref{Morrey}),
\begin{align*}
\|\xi (\ak)_i^j v_{\kappa ttt,j}^i - [\xi (\ak)_i^j
v_{ttt,j}^i]_\kappa \|_{0,+} \le C \kappa \|\ak\|_{3,+}
\|v_{ttt}\|_{1,+}\,.
\end{align*}
This inequality together with the ``divergence free'' constraint
implies
\begin{align*}
\|\xi (\ak)^j_i v_{\kappa ttt,j}^i\|_{0.+} \le C\kappa \|\ak\|_{3,+}
\|v_{ttt}\|_{1,+} + C \mP(\Ek(t))
\end{align*}
and therefore by Young's inequality,
\begin{align*}
& \int_0^T \int_\Op \xi (\ak)_i^j q_{tt,\alpha}^+ (v^+ -
w)\cdot\eta_{\kappa,\beta} g_\kappa^{\alpha\beta} v_{\kappa ttt,j}^i
dx dt \\
\le&\ \delta \sup_{t\in[0,T]} \Ek(t) + C(\delta) T\mP(\sup_{t\in
[0,T]} \Ek(t))\,.
\end{align*}
For the former case, we make use of the equation (\ref{Eulerreg}a)
to substitute $(\ak)_i^k q_{,k}$ for $v_t^i$. Therefore, in this
case the worst term is
\begin{align*}
\int_0^T \int_\Op \xi \partial_\alpha[(\ak)_i^j q_{tt,j}^+] (v^+ -
w)\cdot\eta_{\kappa,\beta}g_\kappa^{\alpha\beta}[(\ak)_i^k
q_{tt,j}^+]_{\kappa} dx dt \equiv \int_0^T \int_\Op
\partial_\alpha Q_i Q_{i\kappa} F^\alpha dx dt\,,
\end{align*}
Let $Q_i = (\ak)_i^j q_{tt,j}^+$ and $F^\alpha = \xi (v^+ - w)\cdot
\eta_{\kappa,\beta} g_\kappa^{\alpha\beta}$. By the definition of
horizontal convolution by layers, we find that
\begin{align*}
\int_0^T \int_\Op
\partial_\alpha Q_i Q_{i\kappa} F^\alpha dx dt = \sum_\ell \int_0^T
\int_{[0,1]^3} (\partial_\alpha Q_i)(\theta_\ell) [\rho \star_h
\rho\star_h (Q_i(\theta_\ell))] F^\alpha(\theta_\ell) dy dt \,.
\end{align*}
Since $(\partial_\alpha Q_i)(\theta) = \Theta_\alpha^\gamma
\partial_\gamma (Q_i(\theta))$,
\begin{align*}
&\int_0^T \int_{[0,1]^3} (\partial_\alpha Q_i)(\theta_\ell) [\rho
\star_h \rho\star_h (Q_i(\theta_\ell))] F^\alpha(\theta_\ell) dx dt \\
=&\ \frac{1}{2} \int_0^T \int_{[0,1]^3} \partial_\gamma |\rho\star_h
(Q_i(\theta_\ell))|^2 F^\alpha(\theta_\ell)
(\Theta_\ell)_\alpha^\gamma dy dt + \int_0^T R dt \,,
\end{align*}
where $R = \rho\star_h \Big[F^\alpha(\theta_\ell)
(\Theta_\ell)_\alpha^\gamma \partial_\gamma Q_i(\theta_\ell)\Big] -
F^\alpha(\theta_\ell) (\Theta_\ell)_\alpha^\gamma \rho\star_h
\Big[\partial_\gamma Q_i(\theta_\ell)\Big]$ and by inequality
(\ref{Morrey}), since $\nabla Q_i \sim F_1(\partial\eta_\kappa)
v_{ttt} + F_2(\partial\eta_\kappa,\partial v_\kappa,\partial
v_{\kappa t})\nabla q + F_3(\partial\eta_\kappa,\partial
v_\kappa)\nabla q_t$,
\begin{align*}
\int_0^T |R| dt &\le C \kappa \int_0^T \|F(\theta)
\Theta_\ell\|_{W^{1,\infty}([0,1]^3)} \|\partial
(Q_i(\theta))\|_{L^2([0,1]^3)}dt \\
&\le M_0(\delta) + \delta \sup_{t\in[0,T]} \Ek(t) + C T
\mP(\sup_{t\in[0,T]}\Ek(t))\,.
\end{align*}
Integrating by parts in space,
\begin{align*}
\int_0^T \int_{[0,1]^3} \partial_\gamma |\rho\star_h
(Q_i(\theta_\ell))|^2 F^\alpha(\theta_\ell)
(\Theta_\ell)_\alpha^\gamma dy dt \le C T \mP(\sup_{t\in[0,T]}
\Ek(t))\,.
\end{align*}
Combining all the estimates above, we find that
\begin{align}
\mI_4 \le M_0(\delta) + \delta \sup_{t\in[0,T]} \Ek(t) + C(\delta) T
\mP(\sup_{t\in[0,T]}\Ek(t))\,. \label{mI4}
\end{align}
Now we turn our attention to $\mI_{22}$ before estimating $\mI_3$.
By the ``divergence free'' constraint (\ref{Eulerreg}b),
\begin{align*}
\int_0^t \mI_{22} dt = \sum_{\sign = \pm} \int_0^t
\int_{\Omega^\sign} q_{ttt}^\sign \Big[ (a_{\kappa ttt})^j_i
v_{,j}^{\sign i} + 3 (a_{\kappa tt})^j_i v_{t,j}^{\sign i} + 3
(a_{\kappa t})^j_i v_{tt,j}^{\sign i} \Big] dx dt\,.
\end{align*}
As shown in \cite{CoSh2007}, it follows from integrating by parts in
time that
\begin{align}
\int_0^T \int_{\Omega^\pm} (a_{\kappa tt})^j_i v_{t,j}^{\pm i}
q_{ttt}^\pm dx ds \le \delta \sup_{t\in[0,T]} \Ek(t) + M_0(\delta) +
CT\mP(\sup_{t\in[0,T]} \Ek(t))\,. \label{I22m}
\end{align}
For the first and the third term, we follow \cite{CoSh2007} and obtain
\begin{align*}
&\sum_{\sign = \pm} \int_0^T \int_{\Omega^\sign} q_{ttt}^\sign \Big[
(a_{\kappa ttt})^j_i v_{,j}^{\sign i} + 3 (a_{\kappa t})^j_i
v_{tt,j}^{\sign i} \Big] dx dt \\
\le& \sum_{\sign = \pm} \int_0^T \int_{\Omega^\sign} \mJ_\kappa^{-1}
(\ak)^r_s (\ak)_i^j \Big[v_{\kappa tt,r}^s v_{,j}^{\sign i} + 3
v_{\kappa,r}^s v^{\sign i}_{tt,j} \Big] q_{ttt}^\sign dx dt \qquad(\equiv \mI_{22_a})\\
&- \sum_{\sign = \pm} \int_0^T \int_{\Omega^\sign} \mJ_\kappa^{-1}
(\ak)_i^r (\ak)_s^j\Big[v_{\kappa tt,r}^s v_{,j}^{\sign i} + 3
v_{\kappa,r}^s v^{\sign i}_{tt,j} \Big] q_{ttt}^\sign dx dt \quad(\equiv \mI_{22_b})\\
& + \delta \sup_{t\in[0,T]} \Ek(t) + M_0(\delta) +
CT\mP(\sup_{t\in[0,T]} \Ek(t)) + C(\delta) \Big[\|v^+_t\|^2_{2.5,+} + \|v^+\|^2_{3.5,+} \\
& + \|\eta_e\|^2_{4.5,+} + \int_0^T \|\sqrt{\kappa}
v^+_{tt}\|^2_{2.5,+} dt \Big]\,.
\end{align*}
Using the ``divergence free'' constraint again,
\begin{align}
\mI_{22_a} &= -3 \sum_{\sign = \pm} \int_0^T \int_{\Omega^\sign}
\mJ_\kappa^{-1} (\ak)^r_s v_{\kappa,r}^s\Big[(a_{\kappa tt})_i^j
v_{,j}^{\sign i} + 2 (a_{\kappa t})^j_i v_{t,j}^{\sign i} \Big]
q_{ttt}^\sign dx dt
\nonumber\\
&\le \delta \sup_{t\in[0,T]} \Ek(t) + M_0(\delta) +
CT\mP(\sup_{t\in[0,T]} \Ek(t))\,, \label{I22a}
\end{align}
where we apply estimates similar to (\ref{I22m}) again from
\cite{CoSh2007}.

Integrating by parts in time (and space if there is $v_{\kappa ttt}$
or $v_{ttt}$), since $\ak = \id$ on $\Go$ and $v_\kappa = 0$ outside
$\Omega'$ (or near $\Go$), we find that
\begin{align*}
\mI_{22_b} \le& \int_0^T \int_{\Gpm} \mJ_\kappa^{-1} (\ak)^r_s
(\ak)_i^j \Big[v_{\kappa ttt}^i v_{,r}^{- s} q_{tt}^-
- v_{\kappa ttt}^i v_{,r}^{+ s} q_{tt}^+ \Big] N_j dS dt \qquad(\equiv \mI_{22_{b_1}})\\
& - 3 \int_0^T \int_{\Gpm} \mJ_\kappa^{-1} (\ak)^r_s (\ak)_i^j \Big[
v_{\kappa,r}^s v^{+ i}_{ttt} q_{tt}^+ - v_{\kappa,r}^s v^{- i}_{ttt}
q_{tt}^- \Big] N_j dS dt \quad(\equiv \mI_{22_{b_2}})\\
& + \delta \sup_{t\in[0,T]} \Ek(t) + M_0(\delta) +
CT\mP(\sup_{t\in[0,T]}\Ek(t))\,,
\end{align*}
where similar estimates for the lower order terms are obtained as
those in \cite{CoSh2007}. It follows from (\ref{trace}) and (\ref{vneq})
that
\begin{align}
\mI_{21_{b_1}} &\le \int_0^T |v_{\kappa ttt}|_{-0.5} \mP(\Ek(t))
dt \le CT\mP(\sup_{t\in[0,T]} \Ek(t))\,, \\
\mI_{22_{b_2}} &= -3 \int_0^T \int_\Gpm \sqrt{\gk} \mJ_\kappa^{-1}
(\ak)_s^r \Big[(v_{ttt}^+ - v_{ttt}^-) \cdot \nk q_{tt}^- + (q^+ -
q^-)_{tt} (v_{ttt}^+ \cdot \nk) \Big] dS dt \nonumber\\
&\le \delta \sup_{t\in[0,T]} \Ek(t) + M_0(\delta) + C T
\mP(\sup_{t\in[0,T]} \Ek(t))\,, \label{I22b2}
\end{align}
where we use the boundary condition (\ref{Eulerreg}c) in the second
term and apply the same estimates as in \cite{CoSh2007}.

For $\mI_3$, we use the boundary condition (\ref{Eulerreg}d) in
$\mI_3$ and obtain
\begin{align*}
\mI_3 &= -\int_0^T \int_\Gpm \Big[\sigma \Delta_g(\eta_e)\cdot \nk +
\kappa \Delta_0 (v^+\cdot\nk)\Big]_{tt} \sqrt{\gk} (v^+ - v^-) \cdot
\partial_t^4\nk dS dt \\
&= \mI_{31} + \mI_{32}\,.
\end{align*}
The worst term of $\mI_3$ is when the time derivatives hit the
highest order term. Since $\displaystyle{\int_0^T
\Big[\|\sqrt{\kappa}v^+_{ttt}\|^2_{1.5,+} + \|\sqrt{\kappa}
v^+_{tt}\|^2_{2.5,+}\Big]dt \le \Ek(T)}$, by Young's inequality,
\begin{align}
\mI_{32} \le CT \mP(\sup_{t\in[0,T]}\Ek(t)) + \int_0^T\Big[\delta
\|\sqrt{\kappa} v_{ttt}^+\|^2_{1.5,+} + C(\delta) \|\sqrt{\kappa}
v_{tt}^+\|^2_{2.5,+}\Big] dt \label{mI31}.
\end{align}
Integrating by parts in time,
\begin{align*}
\mI_{31} =& \int_0^T \int_\Gpm \sigma \Big[\Delta_g(\eta_e)\cdot
\nk\Big]_{ttt} \sqrt{\gk} (v^+ - v^-)\cdot \partial_t^3 \nk dS dt \qquad(\equiv \mI_{{31}_a}) \\
& - \int_\Gpm \sigma \Big[\Delta_g(\eta_e)\cdot \nk\Big]_{tt}
\sqrt{\gk} (v^+ - v^-)\cdot \partial_t^3 \nk dS\Big|_{t=0}^{t=T}
\qquad(\equiv \mI_{{31}_b})\,.
\end{align*}
For $\mI_{{31}_a}$, it follows from integration by parts (in space)
that
\begin{align*}
\mI_{{31}_a} 
\le& - \int_0^T\int_\Gpm \sigma g^{\gamma\delta}
(v^+_{tt,\gamma})\cdot \nk \sqrt{\gk} (v^+ -
v^-)\cdot\eta_{\kappa,\beta}g_\kappa^{\alpha\beta} (v_{\kappa
tt,\alpha\delta}\cdot \nk) dS dt \\
& + CT\mP(\sup_{t\in[0,T]} \Ek(t)) \,.
\end{align*}
By the definition of $v_\kappa$, the inequality above implies that
\begin{align*}
\mI_{{31}_a} \le& - \int_0^T\int_\Gpm \sigma [\partial_\gamma (\rho
\star_h v^+_{tt})\cdot \nk] F^{\alpha\gamma\delta} [\partial_\delta
(\rho \star_h v^+_{tt}) \cdot \nk]_{,\alpha} dS dt \\
& + CT\mP(\sup_{t\in[0,T]} \Ek(t))\,,
\end{align*}
where $F^{\alpha\gamma\delta} = \sqrt{\gk}
g_\kappa^{\alpha\beta}g^{\gamma\delta} (v^+ -
v^-)\cdot\eta_{\kappa,\beta}$. Since $F^{\alpha\gamma\delta}$ is
symmetry in $\gamma$ and $\delta$, it follows from integration by
parts that
\begin{align*}
\mI_{{31}_a} &\le \frac{1}{2} \int_0^T\int_\Gpm \sigma
(\partial_\gamma v^+_{tt}\cdot \nk) F^{\alpha\gamma\delta}_{,\alpha}
(\partial_\delta v^+_{tt} \cdot \nk) dS dt + CT\mP(\sup_{t\in[0,T]}
\Ek(t))
\end{align*}
and hence
\begin{align}
\mI_{{31}_a} \le CT\mP(\sup_{t\in[0,T]} \Ek(t))\,. \label{mI31a}
\end{align}
Integrating by parts in space, the worst term of $\mI_{{31}_b}$ is
\begin{align*}
- \int_\Gpm \sigma F^{\alpha\gamma\delta} (\partial_\gamma
v^+_t\cdot \nk) (\partial_\delta v_{\kappa tt}\cdot\nk)_{,\alpha}
dS\,.
\end{align*}
Since $F^{\alpha\gamma\delta}_t\in L^2(0,T;L^\infty(\Gpm))$,
integrating by parts in space, we find that
\begin{align}
\mI_{{31}_b} \le \delta \sup_{t\in[0,T]} \Ek(t) + C(\delta)
\|v_t^+\|^2_{2.5,+} + C T \mP(\sup_{t\in[0,T]} \Ek(t))\,.
\label{mI31b}
\end{align}
Combining all the estimates above,
\begin{align}
&\sup_{t\in[0,T]}\Big[\|v_{ttt}\|^2_{0,\pm} + |v^+_{tt}\cdot
n|^2_1\Big] + \int_0^T |\sqrt{\kappa}
v^+_{ttt}\cdot\nk|^2_1 dt \nonumber\\
\le&\ M_0(\delta) + \delta \sup_{t\in[0,T]} \Ek(t) + C(\delta) T
\mP(\sup_{t\in[0,T]} \Ek(t)) + C(\delta) \Big[\|v^+_t\|^2_{2.5,+} \label{t3d0ineqtemp} \\
& + \|v^+\|^2_{3.5,+} + \|\eta_e\|^2_{4.5,+} + \int_0^T
\|\sqrt{\kappa} v^+_{tt}\|^2_{2.5,+} dt\Big] \nonumber\,.
\end{align}

We also need controls for $|v^-_{tt}\cdot n|_1$. It follows from
inequality (\ref{Morrey}) and the fundamental theorem of calculus
that
\begin{align*}
|w\cdot(n-\nk)|_1 &\le C \kappa \Big[M_0 +
CT\mP(\sup_{t\in[0,T]}\Ek(t))\Big] |w|_1 |\eta^+|_{4.5}\,.
\end{align*}
Therefore, by (\ref{Eulerreg}e) and the fundamental theorem of
calculus,
\begin{align*}
& |v^-_{tt}\cdot n|_1 \le |v^-_{tt}\cdot\nk|_1 + |v^-_{tt}\cdot (n-\nk)|_1 \\
\le&\ |(v^+_{tt} - v^-_{tt})\cdot\nk|_1 + |v^+_{tt}\cdot\nk|_1
+ |v^-_{tt}\cdot (n-\nk)|_1\\
\le&\ |2(v_t^+ - v_t^-)\cdot n_{\kappa t} + (v^+ - v^-)\cdot
n_{\kappa tt}|_1 + |v_{tt}^+\cdot n|_1
+ |v_{tt}\cdot (n - \nk)|_{1,\pm} \\
\le&\ M_0(\delta) + \delta \sup_{t\in[0,T]} \Ek(t) +
CT\mP(\sup_{t\in[0,T]}\Ek(t)) + |v_{tt}^+\cdot n|_1\,.
\end{align*}
Having this additional inequality, we find that
\begin{align}
& \sup_{t\in[0,T]}\Big[\|v_{ttt}\|^2_{0,\pm} + |v_{tt}\cdot
n|^2_{1,\pm}\Big] + \int_0^T
|\sqrt{\kappa} v^+_{ttt}\cdot\nk|^2_1 dt \nonumber\\
\le&\ M_0(\delta) + \delta \sup_{t\in[0,T]} \Ek(t) + C(\delta) T
\mP(\sup_{t\in[0,T]} \Ek(t)) + C(\delta) \Big[\|v^+_t\|^2_{2.5,+} \label{t3d0ineq} \\
& + \|v^+\|^2_{3.5,+} + \|\eta^+\|^2_{4.5,+} + \int_0^T
\|\sqrt{\kappa} v^+_{tt}\|^2_{2.5,+} dt\Big] \nonumber\,.
\end{align}

\subsection{Estimates for the second time-differentiated
$\kappa$-problem} Similar to (12.33) in \cite{CoSh2007}, let
$\xi\partial\partial_t^2$ act on (\ref{Eulerreg}b) and test against
$\xi\partial v_{tt}$, we find that for $\delta_1 > 0$,
\begin{align}
&\sup_{t\in[0,T]} |\partial^2 v_t\cdot n|^2_{0,\pm}  + \int_0^T
|\sqrt{\kappa}
\partial^2 v^+_{tt}\cdot\nk|^2_{0,\pm}
dt \le M_0(\delta_1) \label{t2d1ineq}\\
&\qquad + \delta_1 \sup_{t\in[0,T]} \Ek(t) + C(\delta_1) T
\mP(\sup_{t\in[0,T]} \Ek(t)) + C(\delta_1) \int_0^T \|\sqrt{\kappa}
v^+_t\|^2_{3.5,+} dt \nonumber \,.
\end{align}
\subsection{Estimates for the time-differentiated
$\kappa$-problem} Let $\xi\partial^2\partial_t$ act on
(\ref{Eulerreg}b) and test against $\xi\partial^2 v_t$, we find that
for $\delta_2>0$,
\begin{align}
&\sup_{t\in[0,T]} |\partial^3 v\cdot n|^2_{0,\pm} + \int_0^T
|\sqrt{\kappa} \partial^3 v^+_t\cdot\nk|^2_{0,\pm} dt \Big\} \le M_0(\delta_2) \label{t1d2ineq}\\
&\qquad + \delta_2 \sup_{t\in[0,T]} \Ek(t) + C(\delta_2) T
\mP(\sup_{t\in[0,T]} \Ek(t)) + C(\delta_2) \int_0^T \|\sqrt{\kappa}
v^+\|^2_{4.5,+} dt  \nonumber\,.
\end{align}
\subsection{The third tangential space differentiated $\kappa$-problem}
Similar to (12.37) in \cite{CoSh2007}, the study of the boundary
condition (\ref{Eulerreg}d) leads to the following important
elliptic
estimate: 
\begin{align}
\sup_{t\in[0,T]} |\sqrt{\kappa}\eta^+ (t)|^2_{5,\pm}
\le M_0 + C \sup_{t\in[0,T]} \Ek(t) + C T
\mP(\sup_{t\in[0,T]}\Ek(t))\,. \label{elliptic}
\end{align}
Let $\xi\partial^3$ act on (\ref{Eulerreg}) and test against
$\xi\partial^3 v$, by (\ref{elliptic}), we find that for
$\delta_3>0$,
\begin{align}
&\sup_{t\in[0,T]} |\partial^4 \eta^+\cdot n|^2_{0,\pm} + \int_0^T
|\sqrt{\kappa} \partial^4 \eta^+\cdot\nk|^2_{0,\pm} dt \nonumber\\
\le&\ M_0(\delta_3) + \delta_3 \sup_{t\in[0,T]} \Ek(t) 
+ C(\delta_3) T \mP(\sup_{t\in[0,T]} \Ek(t)) \label{t0d3ineq} \,. 
\end{align}


\subsection{A polynomial-type inequality for the energy and the existence of solutions}
Combining the div-curl estimates (\ref{divcurlp}), (\ref{divcurlm}),
the energy estimates (\ref{t3d0ineq}), (\ref{t2d1ineq}),
(\ref{t1d2ineq}), (\ref{t0d3ineq}), we find that
\begin{align*}
\Ek(t) \le&\ M_0(\delta,\delta_1,\delta_2,\delta_3) + (\delta+
\delta_1 C(\delta) + \delta_2 C(\delta_1) + \delta_3 C(\delta_2))
\sup_{t\in[0,T]} \Ek(t) \\
&+ C(\delta,\delta_1,\delta_2,\delta_3) T \mP(\sup_{t\in[0,T]}
\Ek(t))\,.
\end{align*}
Choose $\delta>0$ and $\delta_j>0$ small enough so that
$\displaystyle{\delta + \delta_1 C(\delta) + \delta_2 C(\delta_1) +
\delta_3 C(\delta_2) \le \frac{1}{2}}$, then the inequality above
implies
\begin{align}
\sup_{t\in[0,T]} \Ek(t) \le M_0 + C T \mP(\sup_{t\in[0,T]}
\Ek(t))\,. \label{polytype}
\end{align}
Therefore, there exists $T_1>0$ independent of $\kappa$ so that
\begin{align}
\sup_{t\in[0,T_1]} \Ek(t) \le 2 M_0\,. \label{kappaindep}
\end{align}
This $\kappa$-independent estimate guarantees the existence of a
solution to problem (\ref{Euler}) by passing $\kappa\to 0$.

\subsection{Removing the additional regularity assumptions on the
initial data}\label{initial} In the previous sections, we in fact
assume that $v$ is smooth enough so that we can directly
differentiate the Euler equation (\ref{Eulerreg}b) and test with
suitable test functions. This requires higher regularity of the
initial data, namely, $u_0^\pm\in H^{10.5}(\Omega^\pm)$ and $\Gpm\in
H^7$. As in \cite{CoSh2007}, this can be achieved by mollifying the
interface by the horizontal convolution by layers and mollifying the
initial velocity by the usual Fredrich's mollifiers.

\subsection{A posteriori elliptic estimates}\label{aposteriori} As in \cite{CoSh2007}, by
exactly the same proof, we find that for $T$ sufficiently small,
\begin{align}
\sup_{t\in[0,T]} [|\Gpm(t)|_{5.5} + \|v\|_{4.5,\pm} +
\|v_t\|_{3,\pm}] \le {\mathcal M}_0\,, \label{aposterioriestimate}
\end{align}
where ${\mathcal M}_0$ is some polynomial of $M_0$.

\section{Optimal regularity for the initial data}\label{optimal}
In the previous discussion, the existence of the solution requires
the initial data $u_0^\pm\in H^{4.5}(\Omega^\pm)$. We show that this
requirement can be loosened to $u_0^\pm\in H^3(\Omega^\pm)$ and
$\Gpm\in H^4$ in this section, by assuming that we already have a
solution to the problem.

In this section, we study the problem in the Eulerian framework. To
start the argument, we define the energy function $\E(t)$ first. Let
$\E(t)$ be define by
\begin{align*}
\E(t) =&\ |\Gpm(t)|^2_4 +
\|u^+\|^2_{H^3(\Op(t))} +
\|u^-\|^2_{H^3(\Om(t))} + \|u_t^+\|^2_{H^{1.5}(\Op(t))} \\
& + \|u_t^-\|^2_{H^{1.5}(\Om(t))} + \|u_{tt}^+\|^2_{L^2(\Op(t))} +
\|u_{tt}^-\|^2_{L^2(\Om(t))} \,.
\end{align*}
Then for the pressure function $p^\pm$, we have the following
estimate:
\begin{align*}
\|p^+\|^2_{H^{2.5}(\Op(t))} + \|p^-\|^2_{H^{2.5}(\Om(t))} +
\|p^+_t\|^2_{H^1(\Op(t))} + \|p^-_t\|^2_{H^1(\Om(t))} \le
C\mP(\E(t))\,.
\end{align*}

The estimates for $\curl u^\pm$ are essentially identical, while the
estimates for $\div u^\pm$ are trivial because of the divergence
free constraint (\ref{Euler}b). Therefore,
\begin{align}
\sup_{t\in[0,T]} &\Big[\|\curl u^+\|^2_{H^{2.5}(\Op(t))} + \|\curl
u^-\|^2_{H^{2.5}(\Om(t))} + \|\curl u^+_t\|^2_{H^1(\Op(t))} \nonumber\\
&+ \|\curl u^-_t\|^2_{H^1(\Om(t))} \|\div u^+\|^2_{H^{2.5}(\Op(t))}
+ \|\div u^-\|^2_{H^{2.5}(\Om(t))} \nonumber\\
& + \|\div u^+_t\|^2_{H^1(\Op(t))} + \|\div
u^-_t\|^2_{H^1(\Om(t))}\Big] \label{divcurluest}\\
\le M_0(\delta) &+ \delta \sup_{t\in[0,T]}\E(t) +
CT\mP(\sup_{t\in[0,T]} \E(t)) \nonumber
\end{align}
where $M_0(\delta) =
M_0(|\Gamma|^2_4,\|u_0^+\|^2_{3,+},\|u_0^-\|^2_{3,-},\delta)$.

\begin{remark} \label{transportrequirement}
The reason for not analyzing the problem in the ALE formulation is
that in the minus region, the transported velocity $a^T(v^- -
v_e^-)$ is only as regular as $\nabla\eta^+$, which is less regular
than the velocity $v^\pm$. This prevents from obtaining the
estimates for $\curl v^\pm$ in $H^{2.5}(\Omega^\pm)$. With the
Eulerian formulation, the transport velocity $u^\pm$ is
$H^3(\Omega^\pm(t))$, and the analysis goes through.
\end{remark}\vspace{.1in}

Note that the need of that $\eta$ is more regular than $u$ (or $v$)
is only for the study of the $\kappa$-problem, where the estimate of
the boundary integrals with artificial viscosity $\kappa$ requires
that $\eta_\kappa$ is as regular as $\sqrt{\kappa} v$ (which is
``more regular'' than $v$ by the definition of the energy function
$\Ek$). This observation implies that without worrying about the
$\kappa$-terms, the energy estimates still follow. Therefore, by the
identities
\begin{align*}
\frac{d}{dt} \int_{\Omega^\pm(t)} f(y,t) dy = \int_{\Omega^\pm(t)}
(f_t + \nabla_{u^\pm} f)(y,t) dy
\end{align*}
and
\begin{align*}
\int_{\Gpm(t)} f(y,t) dS_y = \int_\Gpm f(\eta(x,t),t) \sqrt{g}
dS_x\,,
\end{align*}
we can show, as shown in the previous sections, that
\begin{align}
&\sup_{t\in[0,T]}\Big[ \|u_{tt}^+\|^2_{L^2(\Op(t))} +
\|u_{tt}^-\|^2_{L^2(\Om(t))} + |\partial^2 v \cdot n|^2_{0,\pm}
+ |\partial v_t\cdot n|^2_{0,\pm}\Big] \label{energyestu}\\
\le&\ M_0(\delta) + \delta \sup_{t\in[0,T]} \E(t) +
CT\mP(\sup_{t\in[0,T]} \E(t)) \,, \nonumber
\end{align}
where the interior estimates are for the Eulerian velocity $u^\pm$
while the boundary estimates are for the ALE velocity $v^\pm$.

In addition to $|\Gpm(t)|^2_4$, it suffices to establish bounds for
$|\partial^2 u^\pm\cdot m|^2_{H^{0.5}(\Gpm(t))}$ and $|\partial
u_t^\pm\cdot m|^2_{L^2(\Gpm(t))}$, where $m$ denotes the unit
outward normal of $\Op(t)$. We remark here that we use different
notations to distinguish the ``normal'' on $\Gpm$ and the normal on
$\Gpm(t)$. In general, $n=m\circ\eta$.

The bounds for $|\partial u_t^\pm\cdot m|^2_{L^2(\Gpm(t))}$ follows
from the energy estimate (\ref{energyestu}). Since
\begin{align*}
u_{t,j}^{\pm i} m^i = \Big[a_j^k v_{t,k}^{\pm i} n^i - a_j^k (v^{\pm
m} a_m^\ell v_{,\ell}^{\pm i})_{,k} n^i\Big]\circ\eta^{-1}
\quad\text{on \ $\Gpm(t)$}\,,
\end{align*}
multiplying $\tau_\alpha^j =
\Big(\frac{\eta^j_{,\alpha}}{|\eta_{,\alpha}|}\Big)\circ\eta^{-1}$
on both side, by $\|\delta - a\|_{2,+} \sim {\mathcal O}(t)$ we find
that
\begin{align}
|\partial_\alpha u^\pm_t \cdot m|_{L^2(\Gpm(t))} &\le C
\Big[|\partial_\alpha v_t^\pm\cdot n|_0 + |\partial_\alpha (v^{\pm
\ell} v^{\pm i}_{,\ell}) n^i|_0 \Big] +
CT\mP(\sup_{t\in[0,T]} \E(t)) \nonumber\\
&\le M_0(\delta) + \delta \sup_{t\in[0,T]} \E(t) +
CT\mP(\sup_{t\in[0,T]} \E(t))\,. \label{utdotmest}
\end{align}

For the bound of $|\partial^2 u^\pm\cdot m|^2_{H^{0.5}(\Gpm(t))}$,
we first estimate $|\partial^2 v^\pm\cdot n|^2_{0.5}$. Similar to
the a posteriori estimate in \cite{CoSh2007}, by studying the boundary
condition
\begin{align*}
\partial_t [(p^+ - p^-)\circ\eta^+ \cdot n] = -[\sqrt{g}
g^{\alpha\beta}\Pi^i_j v^{+j}_{,\beta} + \sqrt{g}
(g^{\nu\mu}g^{\alpha\beta} -
g^{\alpha\nu}g^{\mu\beta})\eta^i_{,\beta} \eta^j_{,\nu}
v^{+j}_{,\mu}]_{,\alpha}\,,
\end{align*}
where $\eta$ and $g$ are formed from $v^+$, we find that
\begin{align*}
|\partial^2 v^+\cdot n|^2_0 &\le M_0(\delta) + \delta |v^+|^2_2 + C
\Big[\mP(|\Gpm|^2_{3.5}, \E(t)) |\eta^+
- \id|^2_2 + |p^+_t - p^-_t|^2_0\Big]\,, \\
|\partial^2 v^+\cdot n|^2_1 &\le M_0(\delta) + \delta |v^+|^2_3 + C
\Big[\mP(|\Gpm|^2_{3.5}, \E(t)) |\eta^+ - \id|^2_3 + |p^+_t -
p^-_t|^2_1\Big]\,,
\end{align*}
and hence by interpolations,
\begin{align*}
\sup_{t\in[0,T]} |\partial^2 v^+\cdot n|^2_{0.5} \le&\ M_0(\delta) +
\delta \sup_{t\in[0,T]} \E(t) + C T
\mP(\sup_{t\in[0,T]} \E(t)) \\
& + C \Big[\|p^+_t\|^2_{H^1(\Op(t))} +
\|p^-_t\|^2_{H^1(\Om(t))}\Big]\,.
\end{align*}
By the elliptic problem
\begin{alignat*}{2}
\Delta p_t^\pm &= 2 \nabla u^\pm_t:(\nabla u^\pm)^T &&\qquad\text{in \
$\Omega^\pm(t)$}\,,\\
\frac{\partial p_t^\pm}{\partial n} &= (u_{tt}^\pm + \nabla_{u^\pm_t}
u^\pm + \nabla_{u^\pm} u^\pm_t)\cdot n &&\qquad\text{on \ $\Gpm(t)$}\,,
\end{alignat*}
we find that
\begin{align}
\|p_t^\pm\|^2_{H^1(\Omega^\pm(t))} &\le C\Big[\|\nabla u^\pm_t:(\nabla
u^\pm)^T\|^2_{H^{0.5}(\Omega^\pm(t))} +
\|u_{tt}^\pm + \nabla_{u_t} u + \nabla_u
u_t\|^2_{L^2(\Omega^\pm(t))}\Big] \nonumber\\
&\le M_0(\delta) + \delta \sup_{t\in[0,T]} \E(t) +
CT\mP(\sup_{t\in[0,T]} \E(t))\,, \label{ptest}
\end{align}
where we use (\ref{energyestu}) to estimate
$\|u_{tt}^\pm\|^2_{L^2(\Omega^\pm(t))}$ and Young's inequality for the
other terms. Therefore,
\begin{align*}
\sup_{t\in[0,T]} |\partial^2 v^+\cdot n|^2_{0.5} \le M_0(\delta) +
\delta \sup_{t\in[0,T]} \E(t) + CT\mP(\sup_{t\in[0,T]} \E(t))\,.
\end{align*}
By similar argument of obtaining (\ref{utdotmest}), we find that
\begin{align}
\sup_{t\in[0,T]} |\partial^2 u^+\cdot m|^2_{H^{0.5}(\Op(t))} \le
M_0(\delta) + \delta\sup_{t\in[0,T]} \E(t) + CT\mP(\sup_{t\in[0,T]}
\E(t))\,. \label{udotmest}
\end{align}
The estimate of $|\partial^2 u^-\cdot m|^2_{H^{0.5}(\Gpm(t))}$ follows
from the boundary condition (\ref{Euler}d), as discussed in the previous
sections.

It then follows from (\ref{divcurluest}), (\ref{energyestu}),
(\ref{utdotmest}) and (\ref{udotmest}) that
\begin{align}
\sup_{t\in[0,T]} \Big[\E(t) - |\Gpm(t)|^2_4\Big] \le M_0(\delta) +
\delta\sup_{t\in[0,T]} \E(t) + CT\mP(\sup_{t\in[0,T]} \E(t))\,.
\label{tempestu1}
\end{align}
With this estimate in mind, we can estimate
$\|p^\pm\|^2_{H^{2.5}(\Omega^\pm(t))}$ in the way we obtain
(\ref{ptest}) and find that $\|p^\pm\|^2_{H^{2.5}(\Omega^\pm(t))}$
satisfies the same inequality. Let $h$ be the height function of
$\Gpm(t)$ over $\Gpm$. By exactly the same argument as in \cite{CoSh2007},
\begin{align*}
\sup_{t\in[0,T]} |h(t)|^2_{H^4} \le M_0(\delta) + \delta\sup_{t\in[0,T]}
\E(t) + CT\mP(\sup_{t\in[0,T]} \E(t))\,,
\end{align*}
and hence
\begin{align}
\sup_{t\in[0,T]}|\Gpm(t)|^2_4 \le M_0(\delta) + \delta \sup_{t\in[0,T]}
\E(t) + CT\mP(\sup_{t\in[0,T]} \E(t))\,. \label{Gpmest}
\end{align}

Combining (\ref{tempestu1}) and (\ref{Gpmest}), by choosing
$\delta>0$ small enough, we obtain the same polynomial-type
inequality as (\ref{polytype}), and therefore there exists a $T>0$
so that
\begin{align*}
\sup_{t\in[0,T]} \E(t) \le 2M_0\,.
\end{align*}
This proves the claim of the optimal regularity of the initial data
to obtain the solution to (\ref{Euler}).

\begin{remark}
The argument in this section can also be used to prove the existence
theorem for the one phase problem studied in \cite{CoSh2007}, provided
the same regularity of the initial velocity $u_0$ and the initial
interface $\Gamma$ are given.
\end{remark}

\section{Uniqueness of solutions} \label{sec::uniqueness}
Suppose that $(v^1,q^1)$ and $(v^2,q^2)$ are both
solutions to (\ref{Eulerreg}) (with $\kappa = 0$, $\mJ = 1$) with
initial data $u_0^\pm\in H^{6}(\Omega^\pm)$ and $\Gpm\in H^{7}$. Let
$\eta_e^1$ and $\eta_e^2$ be defined as in Section \ref{linearkprob}
(with associated cofactor matrices $a^1$ and $a^2$), and set
\begin{align*}
\E_j (t) = \|\Gpm(t)\|^2_{7} + \sum_{k=0}^4 \|\partial_t^k
v^j(t)\|^2_{6-1.5k,\pm} + \sum_{k=0}^3 \|\partial_t^k
q^j(t)\|^2_{5.5-1.5k,\pm} \,.
\end{align*}
By the existence theorem, both $\E_1(t)$ and $\E_2(t)$ are bounded
by a constant ${\mathcal M}_0$ depending on the data $u_0$ and
$\Gamma$ on a time interval $0\le t\le T$ for $T$ small enough.
 
Let $w=v^1-v^2$, $w_e = M^+ w^+$ with associate flow map $\zeta_e
=\int_0^t M^+ w^+ ds$, and $r=q^1-q^2$. The goal in this section is
to show that $w=0$ by showing that the energy function
\begin{align*}
E(t) = \|v(t)\|^2_{3,\pm} + \|v_t(t)\|^2_{1.5,\pm} +
\|v_{tt}(t)\|^2_{0,\pm}
\end{align*}
is actually zero for a short time.
\subsection{The divergence and curl estimates}
In $\Op$, $v^{1+}$ and $v^{2+}$ satisfy
\begin{align*}
\rho^+ v^{+j}_t + a_j^\ell q_{,\ell} = 0 \qquad \text{for \
$(v,q)=(v^1,q^1)$ or $(v,q)=(v^2,q^2)$}\,.
\end{align*}
Let $\varepsilon_{ijk} a_j^r \nabla_r$ act on both sides of the
equality above and form the difference of the two equations, after
integrating in time from $0$ to $t$, we find that
\begin{align*}
& \rho^+ \curl w^{+i}(t) = \varepsilon_{ijk} \int_0^t \Big[ (a^1 -
a^2)_k^\ell (v^1_t)^{+j}_{,\ell} + [(a^2)_k^\ell - \delta_k^\ell]
w_{t,\ell}^{+j}\Big] ds \\
=&\ \varepsilon_{ijk} [(a^2)_k^\ell(t) - \delta_k^\ell]
w_{,\ell}^{+j}(t) + \varepsilon_{ijk} \int_0^t \Big[ (a^1 -
a^2)_k^\ell (v^1_t)^{+j}_{,\ell} - (\partial_t a^2)_k^\ell
w_{,\ell}^{+j}\Big] ds\,.
\end{align*}
Therefore, by $\|a^2(t) - \delta(t)\|_{3.5,+} \le CT$,
\begin{align*}
\sup_{t\in[0,T]} \|\curl w^+(t)\|^2_{2,+} \le CT \sup_{t\in[0,T]}
\|w^+(t)\|^2_{3,+} \,,
\end{align*}
where $C$ depends on ${\mathcal M}_0$ only. By the ``divergence
free'' constraint $a^j_i v^i_{,j} = 0$, we similarly have
\begin{align*}
\sup_{t\in[0,T]} \|\div w^+(t)\|^2_{2,+} \le CT \sup_{t\in[0,T]}
\|w^+(t)\|^2_{3,+}\,.
\end{align*}
 
 
For the divergence and curl estimates in $\Om$, let $\tv^1$ and
$\tv^2$ denote the Lagrangian velocity in $\Om$, that is,
\begin{align*}
\tv^j = \partial_t \tilde{\eta}^j = u^j\circ \tilde{\eta}^j
\qquad\text{in \ $\Om$}\,,
\end{align*}
where $u^j$ is the Eulerian velocity in $\Om$. The same argument as
above shows that
\begin{align}
\sup_{t\in[0,T]} \Big[ \|\curl \tilde{w} (t)\|^2_{2,-} + \|\div
\tilde{w} (t)\|^2_{2,-} \Big] \le CT \sup_{t\in[0,T]}
\|\tilde{w}(t)\|^2_{3,-}\,, \label{tempdivcurl}
\end{align}
where $\tilde{w} = \tv^1 - \tv^2$. We now convert
(\ref{tempdivcurl}) to the inequality with $w$ replacing
$\tilde{w}$.
 
Let $\zeta^j = (\tilde{\eta}^j)^{-1}\circ (\eta^j_e)^{-1}$ and $b^j
= \nabla \zeta^j$ for $j=1,2$. Then
\begin{align*}
\|\curl w(t)\|^2_{2,+} &= \sum_{i=1}^{\mathfrak n}
\|\varepsilon_{ijk} [\tv^1\circ\zeta^1 - \tv^2\circ\zeta^2]_{,k}^j
\|^2_{2,-} = \sum_{i=1}^{\mathfrak n} \|\varepsilon_{ijk} [(b^1)_k^r
\tv_{,r}^{1j} - (b^2)_k^r \tv^{2j}_{,r}]\|^2_{2,-}\\
&= \sum_{i=1}^{\mathfrak n} \|\varepsilon_{ijk} [(b^1 - b^2)_k^r
\tv_{,r}^{1j} + (b^2 - \delta)_k^r (\tv^{1} - \tv^2)^j_{,r} + \tilde{w}^j_k]\|^2_{2,-} \\
&\le C T \sup_{t\in[0,T]} \Big[\|w^+\|^2_{3,+} +
\|\tilde{w}(t)\|^2_{3,-}\Big] \,,
\end{align*}
where we use $\|(b^1 - b^2)(t)\|^2_{2,-} \le C T \sup_{t\in[0,T]}
\Big[ \|w^+\|^2_{3,+} + \|\tilde{w}\|^2_{3,-}\Big]$ by studying the
time derivative of $b^1 - b^2$. Similar argument shows that
\begin{align*}
\|\div w\|^2_{3,-} &\le C T \sup_{t\in[0,T]} \Big[ \|w^+\|^2_{3,+} +
\|\tilde{w}\|^2_{3,-}\Big]\,,\\
\|\tilde{w}\|^2_{3,-} &\le C T \sup_{t\in[0,T]} \Big[
\|w^+\|^2_{3,+} + \|\tilde{w}\|^2_{3,-}\Big]\,.
\end{align*}
Thus for $T>0$ small enough, we find that
\begin{align*}
\sup_{t\in[0,T]} \Big[ \|\curl w (t)\|^2_{2,-} + \|\div w
(t)\|^2_{2,-} \Big] \le CT \sup_{t\in[0,T]} \|w(t)\|^2_{3,-}\,.
\end{align*}
The estimates for the divergence and curl of $w_t$ are similar, so
we omit here. In a nut shell,
\begin{align}
\sup_{t\in[0,T]} \Big[&\|\curl w(t)\|^2_{2,\pm} + \|\div
w(t)\|^2_{2,\pm} + \|\curl w_t(t)\|^2_{0.5,\pm} \nonumber\\
& + \|\div w_t(t)\|^2_{0.5,\pm} \Big] \le CT \sup_{t\in[0,T]}
E(t)\,. \label{divcurlunique}
\end{align}
 
\begin{remark}
We cannot obtain estimate (\ref{divcurlunique}) by studying the
equations for $v^-$ (with the transport velocity) directly, since it
also requires the study of $\zeta_e^+$ as we did in the estimates of
the $\kappa$-problem. This requires $u_0^\pm\in H^5(\Omega^\pm)$ at
least.
\end{remark}
 
\subsection{The boundary estimates}
First we note that $(w,r)$
satisfies
\begin{subequations}\label{unique}
\begin{alignat}{2}
\rho^\pm w^{\pm i}_t + (a^1)_j^\ell (v^{1-} -& v_e^{1-})^j
w_{,\ell}^{- i} + (a^1)^j_i r^\pm_{,j} = F^\pm
&&\qquad\text{in \ $[0,T]\times\Omega^\pm$}\,,\\
(a^1)^j_i w^{\pm i}_{,j} &= (a^2 - a^1)^j_i v^{2 \pm i}_{,j} &&\qquad\text{in \ $[0,T]\times\Omega^\pm$}\,,\\
(r^+ - r^-)n_1 &= - \sigma \Pi^1 g^{1\alpha\beta} \zeta_{e,\alpha\beta} + B &&\qquad\text{on \ $[0,T]\times\Gpm$}\,,\\
w^+ \cdot n_1 &= w^- \cdot n_1 + b &&\qquad\text{on \ $[0,T]\times\Gpm$}\,,\\
w^- \cdot n_1 &= 0 &&\qquad\text{on \ $[0,T]\times\Go$}\,,\\
w^\pm(0) &= 0 &&\qquad\text{in \ $\{t=0\}\times\Omega^\pm$}\,,
\end{alignat}
\end{subequations}
where
\begin{align*}
F^\pm &= \Big[(a^2 - a^1)_j^\ell (v^{2-} - v^{2-}_e)^j - (a^1)_j^\ell (w^{-j}
- w^{-j}_e) \Big](v^{2-}_{,\ell})^i + (a^2 - a^1)_j^\ell q^{2\pm}_{,\ell}\,, \\
B &= -\sigma \Delta_{g^1-g^2}(\eta_e^2) + (q^{2+} - q^{2-}) (n_2 - n_1) \,, \\
b &= (v^{2+} - v^{2-}) \cdot (n_2 - n_1) \,.
\end{align*}
 
The main difference between (\ref{unique}) and the uniqueness
argument for the one phase problem (see Section 15 in \cite{CoSh2007})
is on the additional term $b$. In order to obtain estimate similar
to (\ref{polytype}) (except that in the uniqueness proof, we only
study the second time differentiated problem), we need to estimate
the integral
\begin{align*}
\int_0^T \int_\Gpm r^-_{tt} (w^+_{tt} - w^-_{tt})\cdot n_1 dS dt\,.
\end{align*}
By (\ref{unique}e),
\begin{align*}
(w^+_{tt} - w^-_{tt})\cdot n = b_{tt} - 2 (w^+_t - w^-_t)\cdot n_t -
(w^+ - w^-) n_{tt}\,.
\end{align*}
The only term we need to worry about is the integral with integrand
$r^-_{tt} (v^{2+} - v^{2-})\cdot (n_2 - n_1)_{tt}$. The worst term
of this integral is (after integration by parts in time)
\begin{align*}
&\ - \int_0^T \int_\Gamma r^-_{tt} \Big[g^{1\alpha\beta}
(v^{1+}_{t,\alpha}\cdot n_1)\eta^1_{e,\beta} - g^{2\alpha\beta}
(v^{2+}_{t,\alpha}\cdot
n_2)\eta^2_{e,\beta}\Big] dS dt\\
=& - \int_\Gamma r^-_{t} \Big[g^{1\alpha\beta}
(v^{1+}_{t,\alpha}\cdot n_1)\eta^1_{e,\beta} - g^{2\alpha\beta}
(v^{2+}_{t,\alpha}\cdot n_2)\eta^2_{e,\beta}\Big] dS
\Big|_{t=0}^{t=T}\\
& +\int_0^T \int_\Gamma r^-_{t} \Big[g^{1\alpha\beta}
(v^{1+}_{t,\alpha}\cdot n_1)\eta^1_{e,\beta} - g^{2\alpha\beta}
(v^{2+}_{t,\alpha}\cdot n_2)\eta^2_{e,\beta}\Big]_t dS dt\,.
\end{align*}
Add and subtract terms to form the integrand in terms of $w$,
$\eta^1_e - \eta^2_e$, $n_1 - n_2$ or $g^1 - g^2$. By Young's
inequality, the first term (time boundary term) is bounded by
$(\delta + C(\delta)T)\sup_{t\in[0,T]} E(t)$, where $C(\delta)$
depends on $\delta$ and ${\mathcal M}_0$. For the second term (time
interior term), the worse term occurs when time differentiating
$\partial v_t$. For this worst case, we can transform the surface
integral to the interior interior using the divergence theorem as we
did in (\ref{bdytoint}). Therefore,
\begin{align*}
\int_0^T \int_\Gamma r^-_{t} \Big[g^{1\alpha\beta}
(v^{1+}_{t,\alpha}\cdot n_1)\eta^1_{e,\beta} - g^{2\alpha\beta}
(v^{2+}_{t,\alpha}\cdot n_2)\eta^2_{e,\beta}\Big]_t dS dt \le
(\delta + C(\delta) T) \sup_{t\in[0,T]} E(t)\,.
\end{align*}
The estimates with the addition of the forcing $F$, the right-hand
side of (\ref{unique}c), and $B$ is already done in \cite{CoSh2007}. It
suffices to show that $|\partial^2 w\cdot n_1|_{0.5}$ has the same
bound. However, since
\begin{align*}
|B_t|^2_{0.5} \le C T \sup_{t\in[0,T]} E(t) + C |w|^2_{1.5} \le
(\delta +  C(\delta) T) \sup_{t\in[0,T]} E(t)\,,
\end{align*}
by study the first time derivative of (\ref{unique}c), similar to
the a posteriori estimate, we find that
\begin{align*}
\sup_{t\in[0,T]} |\partial^2 w\cdot n_1|^2_{0.5} \le (\delta +
C(\delta) T) \sup_{t\in[0,T]} E(t)\,.
\end{align*}
Therefore, with (\ref{divcurlunique}) we conclude that $E(t)$
satisfies
\begin{align*}
\sup_{t\in[0,T]} E(t) \le (\delta + C(\delta)T) \sup_{t\in[0,T]}
E(t)\,,
\end{align*}
which implies for $T$ small enough, $E(t) = 0$ and hence $w=0$. In
other words, we establish the uniqueness of the solution to the
problem.

\section*{Acknowledgments}
The research is supported by the National Science Foundation under
grants NSF DMS-0313370 and EAR-0327799.

\end{document}